\documentclass[11pt]{amsart}
\input{Style.sty}

\calclayout

\title{An $L_{\infty}$ structure on symplectic cohomology}
\author{Matthew Strom Borman, Mohamed El Alami, and Nick Sheridan}
\begin{document}
\begin{abstract}
    We construct the $L_\infty$ structure on symplectic cohomology of a Liouville domain, together with an enhancement of the closed--open map to an $L_\infty$ homomorphism from symplectic cochains to Hochschild cochains on the wrapped Fukaya category. Features of our construction are that it respects a modified action filtration (in contrast to Pomerleano--Seidel's construction); it uses a compact telescope model (in contrast to Abouzaid--Groman--Varolgunes' construction); and it is adapted to the purposes of our follow-up work where we construct Maurer--Cartan elements in symplectic cochains which are associated to a normal-crossings compactification of the Liouville domain.
\end{abstract}
\maketitle
\tableofcontents

\section{Introduction}

\subsection{$L_\infty$ structure on Floer cochains: the compact case}

Let $M$ be a compact symplectic manifold. 
We recall the construction of the \emph{Floer cochain complex} of $M$: it is a cochain complex $(CF^*(M;H),\partial)$, whose generators are orbits of an auxiliary Hamiltonian function $H: S^1 \times M \to \R$. 
We also recall the proof that two different choices of auxiliary data give rise to quasi-isomorphic Floer cochain complexes, with the quasi-isomorphism given by a \emph{continuation map}. 
This shows that the cohomology of the Floer cochain complex (the \emph{Hamiltonian Floer cohomology of $M$}) is independent of auxiliary choices, in particular is an invariant of $M$. 
Finally, we recall that this invariant does not contain any new information in itself:\footnote{Of course, the fact that its generators have geometric significance is very interesting; we are just saying that the cohomology group is not an interesting invariant in itself.} the \emph{Piunikhin-Salamon-Schwarz isomorphism} is a quasi-isomorphism from the cohomology of $M$ to its Hamiltonian Floer cohomology, so the Hamiltonian Floer cohomology is just the cohomology of $M$.

However, the Floer cochain complex of $M$ is endowed with an interesting algebraic structure: an action of the chains on the framed little discs operad \cite{Abouzaid-Groman-Varolgunes}. 
This includes, as part of the data, an $L_\infty$ structure which should control the deformations of the structure, see \cite{Fabert}.  
However, one expects the PSS isomorphism to intertwine this structure with the action of the gravity operad on the cohomology of $M$ arising from Gromov--Witten theory, via the map from the framed little discs operad to the gravity operad which forgets the framings. 
The chains in framed little discs which define the $L_\infty$ operations map to degenerate chains in the gravity operad, so one expects the $L_\infty$ structure on the Floer cochain complex to be formal and abelian (and one can prove this by enhancing the PSS isomorphism to an $L_\infty$ quasi-isomorphism). 

This means that the space of solutions to the Maurer--Cartan equation for the $L_\infty$ structure on Floer cochains can be identified with the cohomology of $M$, and in particular is smooth.  
In fact, under mirror symmetry, this formality corresponds to formality of the DG Lie algebra of polyvector fields on the mirror variety: this is at the heart of the Bogomolov-Tian-Todorov \cite{Bogomolov1978,Tian1987,Todorov1989} theorem about unobstructedness of the moduli space of Calabi-Yaus, and of Kontsevich's deformation quantization \cite{Kontsevich-deformation-quantization}. 

There is an $L_\infty$ homomorphism called the closed--open map, mapping from Floer cochains to the Hochschild cochains on the Fukaya category endowed with their natural structure as differential graded Lie algebra:
\begin{equation}
    \label{eq:CO}
\CO:CF^*(M;H) \dashrightarrow CC^*(\mathcal{F}(M)). 
\end{equation}
In good cases this map is a quasi-isomorphism.  With that in mind, the closed-open map \eqref{eq:CO} may be thought of as the A-side analogue of Kontsevich's celebrated formality theorem, see \cite[\S 4.6.2]{Kontsevich-deformation-quantization}. 

When $\CO$ is a quasi-isomorphism, abstract deformations of the Fukaya category are in bijection with Maurer--Cartan elements in $CC^*(\mathcal{F}(M))$, which in turn are in bijection with classes in $H^*(M)$ by formality and abelianness. 
The corresponding deformations of $\mathcal{F}(M)$ are precisely the so-called `bulk deformations'; so when $\CO$ is a quasi-isomorphism, the bulk deformations represent all abstract deformations of $\mathcal{F}(M)$. 

\subsection{The case of Liouville domains}

The purpose of this paper is to construct an $L_\infty$ structure on the symplectic cochains on a Liouville domain $X$, and the closed--open map to the Hochschild cochains on its wrapped Fukaya category. 
In general, this $L_\infty$ structure will be non-formal, in contrast to the case of compact $M$: that is because it receives contributions from non-constant orbits, on which loop rotation acts non-trivially, so that the action of framed little discs does not factor through an action of the gravity operad.

We work throughout over the coefficient ring $\Z$. 
We use the telescope model for symplectic cochains:
    $$SC^*(X) := \bigoplus_{n=1}^\infty CF^*(X,H_n)[t],$$
where $(H_n)_{n=1}^\infty$ is an increasing sequence of Hamiltonians on the completion $\widehat{X}$ of $X$, cofinal among those which are negative on $X$ (\`a la Viterbo \cite{Viterbo}); and $t$ is a formal parameter of degree $-1$ satisfying $t^2=0$. 
The differential $\partial$ is given by
$$\partial(x+ty) = \delta(x) - t\delta(y) + \kappa(y)-y,$$
where $\delta:CF^*(X,H_n) \to CF^*(X,H_n)[1]$ is the usual Floer differential, and $\kappa:CF^*(X,H_n) \to CF^*(X,H_{n+1})$ is a continuation map. 

Our first main construction is that of an $L_\infty$ structure on $SC^*(X)$: this consists of structure maps
$$\ell^d:SC^*(X)^{\otimes d} \to SC^*(X)[3-2d]$$
for $d \ge 1$, with $\ell^1 = \partial$, which are graded-commutative:
$$\ell^d(x_{\sigma(1)},\ldots,x_{\sigma(d)}) = (-1)^\epsilon \ell^d(x_1,\ldots,x_d)$$
and also satisfy the $L_\infty$ relations:
\begin{equation}
    \sum_{\substack{1 \le j \le d\\ \sigma \in \Unsh(j,d)}} (-1)^{\epsilon} \ell^{d-j+1}(\ell^j(x_{\sigma(1)},\dots,x_{\sigma(j)}),x_{\sigma(j+1)},\dots,x_{\sigma(d)}) = 0.
\end{equation}
In the equations above, and the rest of the paper, we use the abbreviation $$\epsilon := \sum_{\substack{i<j \\ \sigma(i)>\sigma(j)}} |x_i||x_j|;$$ and $\Unsh(j,d)$ is the group of $j$-unshuffles, which are permutations $\sigma \in \mathfrak{S}_d$ which satisfy
\begin{equation}
    \sigma(1)<\dots<\sigma(j) \quad\text{and}\quad \sigma(j+1)<\dots<\sigma(d).
\end{equation}

Our setup here differs from more conventional definitions of $L_{\infty}$ structures, but the difference is merely cosmetic. 
For instance, if we shift the grading by $1$ and set
\begin{equation}\label{eq:LM-trans}
    \ell^d_{\mathrm{LM}}(x_1,\dots,x_d) = (-1)^{\sum_{i=1}^d (d-i)\abs{x_i}} \ell^d(x_1,\dots,x_d),
\end{equation}
then the operations $\ell^d_{\mathrm{LM}}$ satisfy the $L_\infty$ relations as defined in \cite{LadaMarkl}.

While $SC^*(X)$ carries an action filtration $\mathcal{A}_{>A} SC^*(X)$, which is respected by the differential $\ell^1$, the higher operations $\ell^d$ do not quite respect it. 
Rather, there is a sequence $(\delta_\nu)_{\nu = 1}^\infty$ of positive numbers tending to $0$, such that the $L_\infty$ operations respect the shifted action filtration $F_{>A}SC^*_\nu(X) := \mathcal{A}_{>A+\delta_\nu}SC^*_\nu(X)$ on the $L_\infty$ subalgebra
$$SC^*_\nu(X):= \bigoplus_{n \ge \nu} CF^*(X,H_n)[t],$$
for all $\nu$.

\begin{rmk}
The $L_\infty$ structure on symplectic cochains is constructed (as part of a larger structure) by Abouzaid--Groman--Varolgunes \cite{Abouzaid-Groman-Varolgunes}; it has also been constructed in a different way by Pomerleano--Seidel \cite{Pomerleano-Seidel}. 
The main advantages of our own construction over these ones are, respectively, that our construction is more explicit than that of Abouzaid--Groman--Varolgunes (who construct their algebraic structures on a much larger, but quasi-isomorphic, chain complex); and admits an action filtration, in contrast to that of Pomerleano--Seidel (who use a quadratic model for symplectic cochains). 
Both of these points are important for the purposes of our followup work \cite{ElAlami-Sheridan-2}. We remark that the action filtration is also important in Siegel's definition of higher symplectic capacities from $L_\infty$ structures in \cite{Siegel}.
\end{rmk}

Our second main construction is that of an $L_\infty$ homomorphism from $SC^*(X)$ to $CC^*(\matheu{W}(X))$, the DG Lie algebra (in particular, $L_\infty$ algebra) of Hochschild cochains on the wrapped Fukaya category $\matheu{W}(X)$, where the latter is constructed using a telescope model as in \cite{Abouzaid13}. 

\begin{rmk}
In \cite{Ritter-Smith}, Ritter and Smith construct the leading term $\CO^1: SC^*(X) \rightarrow \CC^*(\matheu{W})$, called the closed-open map, by counting isolated points in the moduli spaces  $\Mms^\disc_{1,k,\bp,\bw}(\bx,\by)$. Earlier versions of this map go back to Abouzaid in his work \cite{Abouzaid-geometric-criterion} on split-generation of the wrapped Fukaya category. 
The $L_\infty$ version of the closed--open map is also constructed by Pomerleano--Seidel in \cite{Pomerleano-Seidel}, using quadratic Hamiltonians rather than telescope complexes.   
\end{rmk}

This consists of maps
$$\CO^d:SC^*(X)^{\otimes d} \to CC^*(\matheu{W}(X))[2-2d]$$
which are graded-commutative in the same sense as before, and furthermore satisfy
\begin{multline*}
    \sum_{\substack{1 \le j \le d\\ \sigma \in \Unsh(j,d)}} (-1)^{\epsilon} \CO^{d-j+1}(\ell^j(x_{\sigma(1)},\ldots,x_{\sigma(j)}),x_{\sigma(j+1)},\ldots,x_{\sigma(d)}) +  \partial \CO^d(x_1,\ldots,x_d) +\\ \sum_{\substack{1 \le j \le d-1\\ \sigma \in \Unsh(j,d) \\ \sigma(1)<\sigma(j+1)}} (-1)^{\epsilon+\sum_{i=1}^j|x_{\sigma(i)}|} \left[ \CO^j(x_{\sigma(1)},\ldots,x_{\sigma(j)}),\CO^{d-j}(x_{\sigma(j+1)},\ldots,x_{\sigma(d)})\right] = 0,
\end{multline*}
where $\partial$ is the Hochschild differential and $[-,-]$ is the Gerstenhaber bracket. 
We explain in Section \ref{Orientations} how this is equivalent to the conventional definition of an $L_\infty$ homomorphism.

It is expected that the closed-open map is a quasi-isomorphism when $X$ is Weinstein. This expectation is a known theorem provided one assumes that the various versions of the wrapped Fukaya category in the literature are compatible. It can be deduced from the generation result of Ganatra-Pardon-Shende (see \cite[theorem 1.13]{GPS2}), which uses a localization model for $\matheu{W}$, combined with Ganatra's earlier work in \cite[Theorem 1.1]{Ganatra-thesis}, which uses a quadratic model for $\matheu{W}$.

When $\CO$ is a quasi-isomorphism, it induces a bijection between abstract deformations of the wrapped Fukaya category and Maurer--Cartan elements in symplectic cohomology. 
In contrast to the compact case, we do not expect the $L_\infty$ structure on symplectic cohomology to be formal or abelian, and we do not expect the space of Maurer--Cartan elements to be smooth. 
Instead, we expect the $L_\infty$ subalgebra of `constant loops' to be formal and abelian, and hence give rise to a smooth component of the moduli space of Maurer--Cartan elements with tangent space $H^2(X)$. 
This corresponds via $\CO$ to the space of `bulk deformations' of the wrapped Fukaya category. 

On the other hand, as we establish in follow-up work building on the present paper (see \cite{ElAlami-Sheridan-2}), an appropriate normal-crossings compactification of $X$ also induces a family of Maurer--Cartan elements (whose dimension is equal to the number of components of the compactifying divisor). This corresponds via $\CO$ to the family of deformations of the compact Fukaya category given by the `relative Fukaya category', cf. \cite{Seidel-deformations}, so we term them `compactifying deformations'. 
It is expected (see \cite[Remark 1.32]{Sheridan-versality}) that at least sometimes, the open symplectic manifold $X$ is mirror to a singular variety, the bulk deformations correspond under mirror symmetry to locally trivial deformations of the singular variety, and the compactifying deformations correspond to smoothings of the singular variety.

\subsection{Outline}

In Section \ref{sec:moduli} we introduce the domain moduli spaces which appear in our constructions. 
In Section \ref{sec:grading} we explain our conventions for gradings and signs in a general context. In Section \ref{sec:Floer} we construct the $L_\infty$ structure maps $\ell^d$. 
In Section \ref{sec:wrapped} we construct the wrapped Fukaya category and the $L_\infty$ homomorphism maps $\CO^d$. 
Finally Section \ref{Orientations} is devoted to checking the signs in all our formulae.

\textbf{Acknowledgments:} M.E.A. and N.S. are supported by an ERC Starting Grant (award number 850713 -- HMS). N.S. is also supported by a Royal Society University Research Fellowship, the Leverhulme Prize, and a Simons Investigator award (award number 929034).

\section{Deligne-Mumford type spaces}\label{sec:moduli}

\subsection{The $A_{\infty}$ operad}

Let $\mathbb{D}$ be the unit disc in the complex plane. We set $\text{Conf}_{1+k}(\partial \mathbb{D})$ for the space of $1+k$ distinct points $(\zeta_0,\dots,\zeta_k) \in (\partial \mathbb{D})^{1+k}$, and $\text{Conf}^{\ord}_{1+k}(\partial \mathbb{D})\subseteq \text{Conf}_{1+k}(\partial \mathbb{D})$ for the subset of configurations whose labelling agrees with their counter-clockwise ordering. The moduli space of smooth discs with $1+k$ boundary marked points is
\begin{equation}
    \Rms^\disc_{k} = \{ (\zeta_0,\dots,\zeta_k) \in \text{Conf}^{\ord}_{1+k}(\partial \mathbb{D})\} / \Aut(\mathbb{D}).
\end{equation}
In the stable range of $k\geq 2$, $\Rms^\disc_k$ is a smooth manifold of dimension $k-2$. Similarly, one defines the moduli space of smooth genus $0$ curves with $1+k$ marked points as
\begin{equation}
    \Rms^\sph_{1+k} = \{ (z_0,\dots,z_k) \in \text{Conf}_{1+k}(\mathbb{P}^1)  \}/\text{PSL}_2(\mathbb{C}).
\end{equation}

By thinking of the circle $\partial \mathbb{D}$ as the equator of $\mathbb{P}^1$, we obtain an embedding 
\begin{equation}\label{Stasheff-into-DM}
    \Rms^\disc_k\hookrightarrow \overline{\Rms}^\sph_{1+k}
\end{equation}
into the Deligne-Mumford compactification of $\Rms^\sph_{1+k}$. Stasheff's associahedra are the compactifications $\overline{\Rms}^\disc_k$ obtained by taking the closure of the image under the embedding \eqref{Stasheff-into-DM}. It turns out that they are all smooth manifolds with corners. Much like Deligne-Mumford spaces, they also admit stratifications modeled on trees which we now explain.

A \emph{$k$-leafed tree} is a finite, connected, directed, acyclic graph $T$ with $1+k$ semi-infinite edges, such that every vertex $v$ of $T$ has a distinguished incoming edge $e^{\text{in}}(v)$ (pointing towards $v$), and a non-empty set of outgoing edges (pointing away from $v$). There is a unique semi-infinite edge pointing towards the tree, called \emph{the root}. The others, called \emph{leaves}, are labeled using the ordered set $\{1,\dots,k\}$. Two such trees $T$ and $T'$ are the same if there is a directed graph isomorphism between them which preserves the labelling of the leaves. More conveniently, $T$ and $T'$ are the same if they have isotopic embeddings in $\mathbb{R}^3$.

When there exists a planar embedding of $T$ such that the leaf labels agree with their counter-clockwise ordering (starting from the root), we say that the $k$-leafed tree is \emph{ordered} (note that in this case, the isotopy class of the planar embedding is uniquely determined by the leaf labels). Let $\mathscr{T}(k)$ be the set of $k$-leafed trees and $\mathscr{T}^{\ord}(k)\subseteq \mathscr{T}(k)$ the subset of ordered trees.

The set of vertices of $T$ is denoted $V(T)$ and the set of internal edges is denoted $\iE(T)$. By Euler's theorem, $\abs{V(T)}=\abs{\iE(T)}+1$. The tree $T$ is said to be \emph{stable} if 
\begin{equation}
    k_v \defeq \deg(v) - 1 \geq 2\ \quad \text{for all}\ v\in V(T).
\end{equation}
We denote by $\mathscr{T}_{\st}(k)$ (resp. $\mathscr{T}_{\st}^{\ord}(k)$) the set of stable $k$-leafed (resp. ordered) trees. The Stasheff associahedra carry stratifications
\begin{equation}\label{stratification-associahedra}
    \overline{\Rms}^\disc_{k} = \bigcup\ \{\Rms^\disc_{T} \ | \ T\in \mathscr{T}_{\st}^{\ord}(k)\}
\end{equation}
by the smooth manifolds $\Rms^\disc_{T} = \prod_{v\in V(T)} \Rms^\disc_{k_v}$, see \cite[Lemma 9.2]{PL-theory} for a proof. We note the following dimension formula
\begin{equation}
    \dim \Rms^\disc_T = k - 2 - \abs{\iE(T)}.
\end{equation}

\subsection{The $L_{\infty}$ operad}

Aside from their Deligne-Mumford compactifications, the spaces $\Rms^\sph_{1+d}$ also admit another compactification $\underline{\Rms}^\sph_{1+d}$ following (the real version of) a recipe due to Fulton and Macpherson \cite{Fulton-Macpherson}. The space $\underline{\Rms}^\sph_{1+d}$ is the \emph{real} blow-up of $\overline{\Rms}^\sph_{1+d}$ along the normal crossings divisor $\overline{\Rms}^\sph_{1+d}\backslash \Rms^\sph_{1+d}$. In particular, it is a smooth manifold with corners, see \cite[\S 5]{Axelrod-Singer-II} for a similar blow-up construction. 

As noted in \cite[\S 3.2]{Kimura-Stasheff-Voronov}, $\underline{\Rms}^\sph_{1+d}$ is the moduli space of stable complex genus $0$ curves with $1+d$ marked points and decorated nodes. A \emph{node decoration} for $C \in \overline{\Rms}^\sph_{1+d}$ is the datum of an anti-linear isomorphism $\sigma : \mathbb{R}\mathbb{P}(T_{z_1}C_1) \cong \mathbb{R}\mathbb{P}(T_{z_2}C_2)$ whenever $C_1$ and $C_2$ are irreducible components of $C$ such that the points $z_1$ and $z_2$ are identified in $C$. Decorations eliminate the $S^1$-ambiguity which appears when attempting to smooth nodes.

Consider the following moduli space
\begin{equation}
    \overline{\Rms}^{\al}_d = \{ (C,z_0,\dots,z_d;\theta_0) \ | \ C \in \underline{\Rms}^\sph_{1+d} \ \text{and} \ \theta_0 \in \mathbb{R}\mathbb{P}(T_{z_0}C) \},
\end{equation}
where $z_0$ is the first marked point of $C$. This is again a smooth manifold with corners because it is an $S^1$-bundle over $\underline{\Rms}^\sph_{1+d}$. It also comes with a stratification analogous to \eqref{stratification-associahedra} which we now explain.

A \emph{framing} on a smooth genus $0$ curve $C \in \Rms^\sph_{1+d}$ is the choice of a tangent direction $\theta_i\in \mathbb{R}\mathbb{P}(T_{z_i}C)$ for each one of its marked points. The framing is said to be \emph{aligned} if for each $i = 1,\dots, d$, there is an isomorphism $\phi: C\rightarrow \mathbb{P}^1$ such that $\phi(z_0) = \infty$, $\phi(z_i) = 0$, and both $\phi_*(\theta_0)$ and $\phi_*(\theta_i)$ point along the positive real direction. Note that the moduli space of smooth genus $0$ curves with $1+d$ marked points and aligned framings is
\begin{equation}
    \Rms^{\al}_d = \text{Conf}_d (\mathbb{C})/\text{Aff}(\mathbb{C},\mathbb{R}_{>0}),
\end{equation}
where $\text{Aff}(\mathbb{C},\mathbb{R}_{>0}) = \{ z\mapsto az + b \ | \ a\in \mathbb{R}_{>0} \ \text{and} \ b\in \mathbb{C} \}$. This is for instance the point of view taken in \cite[\S 3.2]{Getzler-Jones} and in \cite[\S 3.3]{Kontsevich-operads-motives}.

A framing on a stable genus $0$ curve $C \in \overline{\Rms}^\sph_{1+d}$ is the choice of a tangent direction $\theta_z \in \mathbb{R}\mathbb{P}(T_zC_i)$ at each marked point or node $z$, for every component $C_i$ of $C$. 
Such a framing is said to be aligned if each component $C_i$ is aligned. 
A framing for $C$ uniquely determines decorations for its nodes: if $z_1 \in C_1$ is identified with $z_2 \in C_2$, then the node decoration is the unique anti-linear isomorphism $\sigma : \mathbb{R}\mathbb{P}(T_{z_1}C_1) \cong \mathbb{R}\mathbb{P}(T_{z_2}C_2)$ such that $\theta_{z_1} \mapsto \theta_{z_2}$.

The moduli spaces of stable genus $0$ curves with aligned framings are stratified by smooth manifolds
\begin{equation}
    \overline{\Rms}^{\al}_d = \bigcup\ \{ \Rms^{\al}_T \ | \ T \in \mathscr{T}_{\st}(d) \},
\end{equation}
where $\Rms^{\al}_T = \prod_{v \in V(T)} \Rms^{\al}_{d_v}$, and $d_v$ is the number of outgoing edges of the vertex $v$. An explicit set of charts around each stratum $\Rms^{\al}_T$ is described in \cite[\S 4.1]{Merkulov}. We note the following dimension formula
\begin{equation}
    \dim \Rms_T^{\al} = 2d - 3 - \abs{\iE(T)}.
\end{equation}

\subsection{The $\text{CO}_{\infty}$-operad}

Let $\Rms^\disc_{d,k}$ be the moduli space of smooth discs with $d$ interior marked points and $1+k$ boundary marked points  labeled according to their counter-clockwise ordering. In the stable range of $2d+k\geq 2$, this is a smooth manifold of dimension $2d+k-2$. In \cite[\S 5]{Kontsevich-deformation-quantization}, Kontsevich introduced a Fulton-Macpherson type compactification $\overline{\Rms}^\disc_{d,k}$ for his proof of the formality conjecture. The strata of this compactification are modeled on \emph{2-colored} trees $T\in \mathscr{T}^{\cl}(d,k)$. The coloring datum on a $(d+k)$-leafed $T$ is a partition of its edges $E(T) = E^{\text{dash}}(T) \sqcup E^{\text{solid}}(T)$ into solid and dashed edges, such that
\begin{itemize}
    \item[(1)] All edges flowing from a dashed edge are dashed too.
    \item[(2)] The sub-tree of solid edges is ordered (in particular, the solid leaves are labeled by the set $\{1,\dots,k\}$).
\end{itemize}

Two such trees $T$ and $T'$ are the same if there is an isomorphism $T \cong T'$ preserving the coloring structure and the labelling of the leaves. It may be convenient to picture a $2$-colored tree as a (solid) tree in the $xy$-plane, with (dashed) sprouts in the $zx$-plane, see Figure \ref{2-colored-tree}.

\begin{figure}
    \centering
    \begin{tikzpicture}[x=0.75pt,y=0.75pt,yscale=-1,xscale=1]
        
        \draw    (179,268) -- (270,269) ;
        \draw    (270,269) -- (337,305) ;
        \draw    (337,305) -- (376,343) ;
        \draw    (337,305) -- (399,289) ;
        \draw    (399,289) -- (482,296) ;
        \draw  [dash pattern={on 4.5pt off 4.5pt}]  (270,269) -- (269,218) ;
        \draw  [dash pattern={on 4.5pt off 4.5pt}]  (399,289) -- (398.5,263.5) ;
        \draw  [dash pattern={on 4.5pt off 4.5pt}]  (269,218) -- (281,174) ;
        \draw  [dash pattern={on 4.5pt off 4.5pt}]  (269,218) -- (243,184) ;
        \draw  [dash pattern={on 4.5pt off 4.5pt}]  (398.5,263.5) -- (413,242) ;
        \draw  [dash pattern={on 4.5pt off 4.5pt}]  (281,174) -- (296,119) ;
        \draw  [dash pattern={on 4.5pt off 4.5pt}]  (281,174) -- (264,122) ;
        \draw  [dash pattern={on 4.5pt off 4.5pt}]  (398.5,263.5) -- (380,241) ;
        \draw  [dash pattern={on 4.5pt off 4.5pt}]  (281,174) -- (318,141) ;
        \draw    (337,305) -- (367,262) ;
    \end{tikzpicture}
    \caption{$2$-colored tree without labels}
    \label{2-colored-tree}
\end{figure}
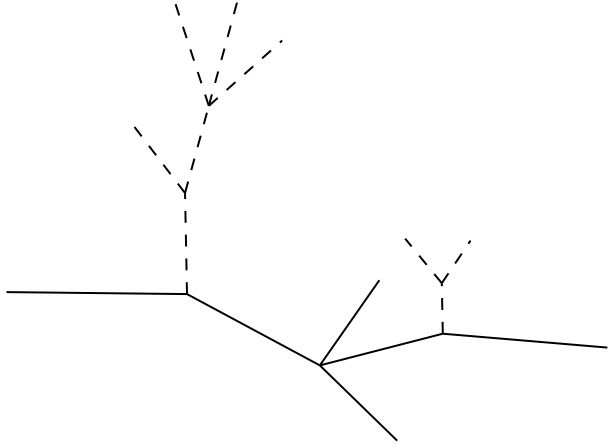

We define $\mathscr{T}_{\st}^{\cl}(d,k)$ to be the set of stable $2$-colored $(d+k)$-leafed trees. Note that the vertices of $T$ acquire colors $V(T) = V^{\text{dash}}(T) \sqcup V^{\text{solid}}(T)$ by declaring that $v\in V^{\text{dash}}(T)$ if and only if its incoming edge is dashed, $e^{\text{in}}(v)\in E^{\text{dash}}(T)$. Kontsevich's compactification is given by
\begin{equation}
    \overline{\Rms}^\disc_{d,k} = \bigcup\ \{ \Rms^\disc_T \ | \ T\in \mathscr{T}_{\st}^{\cl}(d,k) \}.
\end{equation}
Each stratum is a product of previously defined moduli spaces,
\begin{equation}
    \Rms^\disc_T = \prod_{v\in V^{\text{solid}}(T) } \Rms^\disc_{d_v,k_v} \times \prod_{v \in V^{\text{dash}}(T)} \Rms^{\al}_{d_v},
\end{equation}
where $d_v$ the number of outgoing dashed edges and $k_v$ is the number of outgoing solid edges of the vertex $v$. The compactification $\overline{\Rms}^\disc_{d,k}$ is a smooth manifold with corners, see \cite[\S 4.2]{Merkulov} for smooth charts around each stratum of this compactification. We note the following dimension formula
\begin{equation}
    \dim \Rms^\disc_T = 2d+k - 2 - \abs{\iE(T)}.
\end{equation}

\subsection{Flavours and weights} 

In \cite[\S 2]{open-string-analogue}, Abouzaid and Seidel explain how to enrich spaces such as $\overline{\Rms}^\disc_d$, $\overline{\Rms}^{\al}_d$, and $\overline{\Rms}^\disc_{d,k}$, with \emph{flavours} and weights. We briefly review how this works for the moduli spaces $\overline{\Rms}^{\al}_d$ of genus $0$ curves with aligned framings.

Let $\bp : F\rightarrow \{1,\dots,d\}$ be a map from some finite set $F$.
A $\bp$-flavour on a smooth framed curve $(C,z_0,\dots,z_d;\theta_0) \in \Rms^{\al}_d$ is a map $\psi: F\rightarrow \text{Isom}(C,\mathbb{P}^1)$ such that for each $f\in F$,
\begin{equation}
   \psi(f)(z_0) = \infty,\quad \psi(f)(z_{\bp(f)}) = 0,
\end{equation}
and $\psi(f)_*(\theta_0)$ points along the positive real direction. Note that in the terminology of \cite{open-string-analogue}, the map $\bp$ is called a flavour and $\psi$ is called a sprinkle.

The space of smooth $\bp$-flavoured framed curves is denoted $\Rms^{\al}_{d,\bp}$. It has the structure of a principal $\mathbb{R}_{>0}^F$-bundle over $\Rms^{\al}_{d}$. Indeed, for any $f\in F$, the possible choices of $\psi(f)$ differ from one another by a positive scaling of $\mathbb{P}^1$.  In particular, $\Rms^{\al}_{d,\bp}$ is a smooth manifold of dimension $2d+\abs{F}-3$.

The strata of the compactification $\overline{\Rms}^{\al}_{d,\bp}$ are modeled 
on pairs $(T,\bF)$ where $T\in \mathscr{T}(d)$ is a labeled $d$-leafed tree and $\bF = (F_{v})_{v\in V(T)}$ is
a partition of $F$ among the vertices of $T$ such that the following compatibility condition holds: 
\begin{itemize}
    \item[($\star$)] If $f\in F_v$, then $v$ belongs to the path from the root to the leaf $\bp(f)$.
\end{itemize}
Given such a partition, each vertex $v\in V(T)$ acquires a map
\begin{equation}
    \bp_v : F_v \rightarrow \{ \text{outgoing edges of } v\},
\end{equation}
where $\bp_v(f)$ is the unique outgoing edge on the path from $v$ to the leaf $\bp(f)$. The tree $T$ itself need not be stable, but the pair $(T,\mathbf{F})$ must be, i.e. $2d_v+\abs{F_v}\geq 3$ for all $v\in V(T)$. The compactification is given by the union 
\begin{equation}
    \overline{\Rms}^{\al}_{d,\bp} = \bigcup\ \{\Rms_{T,\bF}^{\al}\ | \ T\in \mathscr{T}(d) \ \text{and}\ (T,\bF)\ \text{stable} \}
\end{equation}
over the product spaces
\begin{equation}
 \Rms^{\al}_{T,\bF} = \prod_{v\in V(T)} \Rms^{\al}_{d_v,\bp_v}.
\end{equation}
Each stratum of this compactification has dimension given by
\begin{equation}
    \dim \Rms^{\al}_{T,\bF} = 2d+\abs{F}-3-\abs{\iE(T)}.
\end{equation}

Our spaces can be further enriched with \emph{weights}. These are collections of integers $(n_0,\dots,n_d)$ associated with the marked points of our Riemann surfaces such that
\begin{equation}\label{weight-equation}
    n_0 = n_1 + \dots + n_d + \abs{F}.
\end{equation}
The moduli space $\overline{\Rms}^{\al}_{d,\bp,\bn}$ is just another copy of $\overline{\Rms}^{\al}_{d,\bp}$. However, we will write it as
\begin{equation}
    \overline{\Rms}^{\al}_{d,\bp,\bn} = \bigcup\ \{\Rms_{T,\bF,\bn}^{\al}\ | \ T\in \mathscr{T}(d) \ \text{and}\ (T,\bF)\ \text{stable} \},
\end{equation}
and we write $\Rms_{T,\bF,\bn}^{\al} = \prod_{v\in V(T)} \Rms^{\al}_{\abs{v},\bp_v,\bn_v}$ to indicate that the weights $\bn$ uniquely determine weights $\bn_v$ for the edges adjacent to the vertex $v$ such that
\begin{itemize}
    \item[(i)] The root has weight $n_0$ and the $i^{\text{th}}$-leaf has weight $n_i$.
    \item[(ii)] At each vertex $v$, equation \eqref{weight-equation} holds.
\end{itemize}

To define $\Rms^\disc_{d,k,\bp}$ (and $\Rms^\disc_{k,\bp} \defeq \Rms^\disc_{0,k,\bp}$), we observe that there is an embedding $\Rms^\disc_{d,k} \hookrightarrow \Rms^{\text{al}}_{k+d}$ by `doubling' the disc, choosing the asymptotic marker at the $0$th boundary marked point to point perpendicularly into the disc, and choosing the unique aligned asymptotic markers at the remaining marked points. 
We define a $\bp$-flavour for an element of $\Rms^\disc_{d,k}$ to be a $\bp$-flavour for its image under this embedding. 
We may also define compactified moduli spaces
\begin{align}
    \overline{\Rms}^\disc_{k,\bp}   &= \bigcup\ \{\Rms^\disc_{T,\bF}\ | \ T\in \mathscr{T}^{\ord}(k) \ \text{and}\ (T,\bF)\ \text{stable} \},\\
    \overline{\Rms}^\disc_{d,k,\bp} &=  \bigcup\ \{\Rms^\disc_{T,\bF}\ | \ T\in \mathscr{T}^{\cl}(d,k) \ \text{and}\ (T,\bF)\ \text{stable} \},
\end{align}
and their enrichments with weights, $\overline{\Rms}^\disc_{k,\bp,\bw}$ and $\overline{\Rms}^\disc_{d,k,\bp,\bw}$. The spaces $ \overline{\Rms}^\disc_{k,\bp}$, $\overline{\Rms}^{\al}_{d,\bp}$, and $\overline{\Rms}^\disc_{d,k,\bp}$ are all compact smooth manifolds with corners. They each come with universal curves $\Univ_{k,\bp}^{\disc}$, $\Univ_{d,\bp}^{\al}$ and $\Univ_{d,k,\bp}^{\disc}$, respectively. For instance, the fiber of $\Univ_{d,\bp}^{\al}\rightarrow \overline{\Rms}^{\al}_{d,\bp}$ over $(C,z_0,\dots,z_d;\theta_0)$, is the (nodal) curve $C$ itself. For proofs, one can adapt the arguments in \cite[\S 6]{open-string-analogue} with little to no modification.

We note that we have an action of the symmetry group
\begin{equation}
    \Sym(\bp) = \{ \pi\in \mathfrak{S}(F) \ | \ \bp\circ \pi = \bp \}
\end{equation}
on $\Rms^{\disc}_{k,\bp}$ by $\pi\cdot(r,\psi) \mapsto (r,\psi\circ \pi)$. Similarly, the space $\Rms^{\al}_{d,\bp}$ carries an action 
\begin{equation}\label{group-action}
    (\sigma,\pi)\cdot(r,\psi)\mapsto (\sigma\cdot r, \psi\circ \pi)
\end{equation}
by the larger group
\begin{equation}
    \Sym(d,\bp) \defeq \{(\sigma,\pi) \in \mathfrak{S}_d \times \mathfrak{S}(F)\ |\ \bp(\pi(f)) = \sigma(\bp(f))\},
\end{equation}
where $\sigma$ acts by permuting the labels of the interior marked points.
Finally, the space $\Rms^\disc_{d,k,\bp}$ carries an action as in \eqref{group-action} by the symmetry group 
\begin{equation}
    \Sym(d,k,\bp) \defeq \{(\sigma,\pi) \in \Sym(k+d,\bp) \ | \ \sigma\vert_{\{1,\dots,k\}} = \text{id}\}.
\end{equation}

\subsection{Asymptotic ends}

Let $C$ be a smooth genus $0$ curve, possibly with boundary. A positive/negative cylindrical end at an interior point $z\in C$ is the choice of a holomorphic embedding
\begin{equation}
    \epsilon_{\pm}: Z^{\cl}_{\pm}\rightarrow C\backslash\{z\}
\end{equation}
 of a positive/negative semi-infinite cylinder $Z^{\cl}_{\pm} = \{ (s,t)\in \mathbb{R}\times S^1 \ | \ \pm s \geq 0 \}$
such that $\lim_{\pm s \rightarrow \infty} \epsilon_{\pm}(s,t) = z$.
We will always assume that, under the embedding
\begin{equation}
    Z^{\cl}_\pm \hookrightarrow \mathbb{P}^1,\quad (s,t)\mapsto e^{-2\pi(s+it)},
\end{equation}
the \emph{end} extends to a holomorphic map $\overline{\epsilon}_{\pm}:\mathbb{P}^1\rightarrow C$. 

Similarly, a positive/negative strip-like end at a boundary point $\zeta\in C$ is the choice of a holomorphic embedding
\begin{equation}
    \epsilon_{\pm}: Z^{\op}_{\pm}\rightarrow C\backslash\{\zeta\}
\end{equation}
 of a positive/negative semi-infinite strip $Z^{\op}_{\pm} = \{ (s,t)\in \mathbb{R}\times [0,1] \ | \ \pm s \geq 0 \}$
such that $\lim_{\pm s \rightarrow \infty} \epsilon_{\pm}(s,t) = \zeta$, which also extends to a holomorphic map $\overline{\epsilon}_{\pm}:\mathbb{D}\rightarrow C$ under then embedding
\begin{equation}
    Z^{\op}_\pm \hookrightarrow \mathbb{D},\quad (s,t)\mapsto \frac{e^{\pi(s+it)}-i}{e^{\pi(s+it)}+i}\overset{.}{}
\end{equation}

It is possible to glue pairs of curves $(C_{+},z_+)$ and $(C_{-},z_-)$ when $z_{\pm}$ are interior (resp. boundary) points equipped with positive/negative cylindrical (resp. strip-like) ends $\epsilon_{\pm}$. For each $\delta\in (0,1)$ (called the gluing parameter), the glued curve is
\begin{equation}\label{gluing}
    C_{+}\#_\delta C_{-} = C_{+} \backslash \epsilon_{+}(s\geq \sigma)\cup C_{-} \backslash \epsilon_{-} (s\leq -\sigma)/ \sim,
\end{equation}
where $\sigma = -\log(\delta)$, and the gluing identification is $\epsilon_{+}(s,t)\sim \epsilon_{-}(s-\sigma,t)$ on the annulus $s\in (0,\sigma)$. 

\begin{lemma}\label{gluing framings}
    Suppose that $z \in C_{+}\ (\text{or}\ C_-)$ is equipped with an end $\epsilon$ and that $\delta$ is sufficently small. Then $z$ inherits an end in the glued curve  $C_{+}\#_\delta C_{-}$.
\end{lemma}

\begin{proof}
    We explain the argument in the case where $z$ is an interior point. When $\delta$ is sufficently small, $z$ survives in the glued curve, which we now think of as
    \begin{equation}\label{gluing2}
        C_{+}\#_\delta C_{-} = C_{+} \backslash \overline{\epsilon}_{+}(\abs{z}\leq \delta)\cup C_{-} \backslash \overline{\epsilon}_{-} (\abs{z}\geq 1/\delta)/ \sim,
    \end{equation}
    where the gluing operation is $\overline{\epsilon}_+(z) = \overline{\epsilon}_{-}(z/\delta)$ on the region $\delta\leq\abs{z}\leq 1$.
    The key idea is that any parametrization $\psi_{\pm}:\mathbb{P}^1\rightarrow C_{\pm}$ can be extended to all of $C_{+}\#_\delta C_{-}$. For $\psi_+$, the extension is
    \begin{equation}
        \widehat{\psi}_+(z) =
        \begin{cases}
           \psi_+(z) &\quad\text{if }  \abs{\overline{\epsilon}_{+}^{-1}\psi_+(z)}\geq \delta\\
           \overline{\epsilon}_{-}\left(1/\delta\cdot\overline{\epsilon}_{+}^{-1}\psi_+(z)\right) &\quad\text{if }  \abs{\overline{\epsilon}_{+}^{-1}\psi_+(z)}\leq 1.\\
        \end{cases}
    \end{equation}
    Likewise, the extension for $\psi_{-}$ is
    \begin{equation}
        \widehat{\psi}_{-}(z) =
        \begin{cases}
           \psi_{-}(z) &\quad\text{if }  \abs{\overline{\epsilon}_{-}^{-1}\psi_-(z)}\leq 1/\delta\\
           \overline{\epsilon}_{+}\left(\delta\cdot\overline{\epsilon}_{-}^{-1}\psi_{-}(z)\right) &\quad\text{if }  \abs{\overline{\epsilon}_{-}^{-1}\psi_-(z)}\geq 1.
        \end{cases}
    \end{equation}
\end{proof}

When a smooth genus $0$ curve $C$ (possibly with boundary) comes with marked points $(z_0,\dots,z_d)$, we make choices of a negative asymptotic end $\epsilon_0$ at $z_0$ (called the \emph{output}), and positive asymptotic ends $(\epsilon_i)_{i=1}^d$ for each of the other marked points (called \emph{inputs}). Furthermore, if $(C,z_0,\dots,z_d;\theta_0)\in \Rms^{\al}_d$ is closed and equipped with a framing $\theta_0\in \mathbb{R}\mathbb{P}(T_{z_0} C)$, we also require that $\overline{\epsilon}_i(\infty) = z_0$ and that $(\overline{\epsilon}_i)_*^{-1}(\theta_0)$ points along the positive real direction for $i=0,\dots,d$. This requirement cuts down the freedom in choosing $\epsilon_i$ to a contractible choice: $\mathbb{R}_{>0}\ltimes\mathbb{C}$ for $\overline{\epsilon}_0$ and $\mathbb{R}_{>0}$ for each of $(\overline{\epsilon}_i)_{i=1}^d$. 

Gluing in the presence of flavours and sprinkles is trickier. For example, let $(C,z_0,\dots,z_d;\theta_0)\in \Rms^{\al}_d $ be a genus $0$ curve equipped with a $\bp$-flavour $\psi$. The pre-image $\ell_{z_{\bp(f)}}\defeq \psi(f)^{-1}(\mathbb{R}_{>0})$ is a line on $C$ joining $z_0$ to $z_{\bp(f)}$. This line is independent of the choice of $\psi$; the latter is equivalent to the choice of a point $\psi(f)^{-1}(1)\in \ell_{z_{\bp(f)}}$, which is called a \emph{sprinkle}. The collection of lines $\ell_{z_i}\subseteq C$ varies continuously across the moduli space $\Rms^{\al}_d$, but it is not compatible with the gluing operation. We briefly recall an approach due to Abouzaid and Seidel in \cite[\S 2.4]{open-string-analogue} which resolves this issue.
\begin{definition}
   A \emph{stick} for an input $z_i$ (where $i\in \{1,\dots,d\}$) is an embedding $\tau_i:\mathbb{R} \rightarrow C\backslash\{z_0,z_i\}$ 
   of the form
   \begin{equation}
    \tau_i(s) = \epsilon_i(r_se^{i\alpha(r_s)}),
   \end{equation} 
   where $r_s = e^{-s}$ and $\alpha:\mathbb{R}_{>0}\rightarrow \mathbb{R}$ satisfies
   \begin{equation}
    \begin{cases}
        & \alpha(r) = 0 \quad\text{if } r\ll 1,\\
        & a\sin(\alpha(r)) + \text{Im}(b)/r = 0  \quad\text{if } r\gg 1,\\
        & \lim \alpha(r) = 0 \quad\text{as } r\rightarrow \infty,
    \end{cases}
   \end{equation}
where $(a,b)\in \mathbb{R}_{>0}\ltimes\mathbb{C}$ is the unique pair such that $\epsilon_0^{-1}\epsilon_i (z) = az+b$.
\end{definition}
The first two constraints on $\tau_i$ ensure that the image $Q_i = \text{Im}(\tau_i)$ is a line which agrees with $\overline{\epsilon}_0(\mathbb{R}_{>0})$ near $z_0$ and with $\overline{\epsilon}_i(\mathbb{R}_{>0})$ near $z_i$. The last constraint $\lim \alpha(r) = 0$ ensures that the set of sticks for an input $z_i$ is contractible. Therefore, one can inductively construct a consistent universal choice of sticks $(\tau_i)_{i=1}^d$ across the moduli spaces $\Rms^{\al}_d$ by gluing. Given such a choice, we identify the choice of a flavour $\psi(f)$ with the the choice of a point on the line $Q_i = \text{Im}(\tau_{\bp(f)})$ given by the intersection $Q_i \cap \psi(f)(S^1)$. This setup makes it straightforward to glue flavours and sprinkles.

\begin{lemma}
 There is a consistent universal choice of cylindrical ends for all the moduli spaces $\Rms^{\al}_{d,\bp}$ which is also $\Sym(d,\bp)$-invariant.   
\end{lemma}

\begin{proof}
    The proof follows the inductive argument laid out in \cite[(9g)]{PL-theory}. First, observe that choosing cylindrical ends $\epsilon_i$ on $\Rms^{\al}_{d,\bp}$ is equivalent to choosing a section of a $\Sym(d,\bp)$-equivariant bundle $\mathscr{E}_i\rightarrow \Rms^{\al}_{d,\bp}$. The fiber of this bundle is diffeomorphic to $\mathbb{R}_{>0}\times \mathbb{C}$, which is contractible. Moreover, gluing $\Sym(d_v,\bp_v)$-invariant ends on $(\Rms^{\al}_{d_v,\bp_v})_{v\in V(T)}$ produces $\Sym(d,\bp)$-invariant ends on $\overline{\Rms}^{\al}_{d,\bp}$ on an open neighborhood of $\partial \overline{\Rms}^{\al}_{d,\bp}$. Therefore, gluing produces a choice ends on the quotient space $\overline{\Rms}^{\al}_{d,\bp}/\Sym(d,\bp)$ near the quotient of the boundary. Since the action of $\Sym(d,\bp)$ on the interior $\overline{\Rms}^{\al}_{d,\bp}\backslash \partial \overline{\Rms}^{\al}_{d,\bp}$ is free, this choice can be extended to all of $\overline{\Rms}^{\al}_{d,\bp}/\Sym(d,\bp)$.
    
    Finally, to start the inductive process, we need to chose cylindrical ends on the following building blocks.
    \begin{itemize}
        \item[-] When $d=1$, $F = {1}$: In this case $\overline{\Rms}^{\al}_{1,\bp}$ is a point, and we can choose any pair of admissible ends $\epsilon_0,\epsilon_1$.
        
        \item[-] When $d=2$ and $F = \emptyset$: The action of $\Sym(2)$ on $\Rms^{\al}_{2}\cong S^1$ is free so that we can choose cylindrical ends on the quotient $\Rms^{\al}_{2}/\Sym(2)$ and pull them back to $\Sym(2)$-invariant cylindrical ends on $\Rms^{\al}_{2}$.
    \end{itemize}
    
\end{proof}

The constructions of $\Rms^\disc_{k,\bp}$ and $\Rms^\disc_{d,k,\bp}$ and their compactifications are similar, see \cite[\S 2.4]{open-string-analogue} for more details.

\section{Grading}\label{sec:grading}

\subsection{The Grading datum}\label{subsection:grading datum}
Let $(X^{2n},\omega)$ be a symplectic manifold. The Lagrangian Grassmannian is the fibre bundle $\gra(X) \to X$, whose fiber $\gra_p(X)$ at $p\in X$ is the manifold of linear Lagrangian subspaces of $(T_pX,\omega_p)$. It comes with a tautological vector bundle $\text{Taut} \to \gra(X)$ whose fibre over a point $\rho\in \gra(X)$ is the Lagrangian subspace $\rho$ itself.

Using the homotopy long exact sequence for fiber bundles, we have 
\begin{equation}\label{homotopy-long-exact-sequence}
\dotsi \to \pi_2(X) \xrightarrow{\delta} \pi_1(\gra_p(X)) \to \pi_1(\gra(X)) \to \pi_1(X) \to 0.
\end{equation}
The Maslov class is an isomorphism $\mu: \pi_1(\gra_p (X)) \to \mathbb{Z}$ which satisfies $\mu \circ \delta = 2 c_1(X)$, we refer to \cite[Chapter 2]{mcduff-salamon-introtosymptopology} for its construction.

If we abelianize \eqref{homotopy-long-exact-sequence}, we obtain an exact sequence
\begin{equation}\label{e:ZtoGX0}
\mathbb{Z} \to H_1(\gra(X)) \to H_1(X) \to 0.
\end{equation}
\begin{definition}\label{grading datum}
 A grading datum is a pair of morphisms $\mathbb{G} = \{\mathbb{Z} \to Y \to \mathbb{Z}_2\}$ whose composition is reduction mod $2$.   
\end{definition}
Throughout our work, we use a grading datum with $Y = H_1(\gra(X))$. The map $\mathbb{Z} \to H_1(\gra (X))$ is the one appearing in \eqref{e:ZtoGX0}, and the map $H_1(\gra(X)) \to \mathbb{Z}_2$ is the pairing with $w_1(\text{Taut})$.  

A $\mathbb{G}$-graded module is the same as a $Y$-graded module. If $V$ is a $\mathbb{G}$-graded module, we say that an element $v\in V$ has degree $d\in \mathbb{Z}$ if the $Y$-degree of $v$ is the image of $d$ under the map $\mathbb{Z}\rightarrow Y$. 
A $\mathbb{G}$-graded module becomes a $\mathbb{Z}_2$-graded module via the map $Y \to \mathbb{Z}_2$. 
In particular, we have the usual symmetric monoidal structure on the category of $\mathbb{G}$-graded modules:
\begin{equation}
M_1 \otimes M_2 \cong M_2 \otimes M_1: \quad m_1\otimes m_2 \mapsto (-1)^{\deg(m_1)\deg(m_2)} m_2\otimes m_1.	
\end{equation}

A $\mathbb{G}$-graded $\Z_2$-torsor $(\{a,b\},y)$ is a $\Z_2$-torsor $\{a,b\}$ together with an element $y \in Y$.\footnote{Note that a $\Z_2$-torsor is the same thing as a set with $2$ elements!} 
We can tensor $\G$-graded $\Z_2$-torsors together (adding the elements of $Y$), and we have a symmetric monoidal structure as above.  
A $\mathbb{G}$-graded $\Z_2$-torsor determines a free $\mathbb{G}$-graded $\Z$-module of rank one called the normalization, $|(\{a,b\},y)| := \langle a,b\rangle/(a+b)[-y]$ (concentrated in degree $y$). 
The functor $|-|$ is symmetric monoidal.

A $\mathbb{G}$-graded line is a $\mathbb{G}$-graded one-dimensional real vector space. 
An isomorphism of graded lines is an equivalence class of isomorphisms of graded real vector spaces, where two isomorphisms are equivalent if they differ by a positive scaling. 
The groupoid of $\mathbb{G}$-graded lines is equivalent to the groupoid of $\mathbb{G}$-graded $\Z_2$-torsors. 
In one direction the equivalence sends a torsor to its $\R$-normalization; in the other it sends a line to the torsor whose elements are the two orientations of the line. 
Thus we will freely speak of tensoring a line with a $\Z_2$-torsor, taking the normalization of a line, etc.

\subsection{Grading in Lagrangian Floer theory}

Let $\Sigma$ be a compact oriented surface with boundary. 
Recall the notion of the \emph{boundary Maslov index} $\mu(E, F) \in \mathbb{Z}$ from \cite[\S C.3]{mcduff-salamon-Jholomorphiccurves}
where $E \to \Sigma$ is a symplectic bundle and $F \subset E_{|\partial \Sigma}$ is a Lagrangian subbundle. 
Now consider a pair $(u, \rho)$ where
$u: \Sigma \to X$ is smooth and $\rho: \partial\Sigma \to \gra(X)$ is a lift of $u_{|\partial \Sigma}$.
Note that
$$
	[\rho] \in \ker \Big(H_1(\gra (X)) \to H_1(X)\Big).
$$ 
\begin{lemma}\label{l:IndexToLoop}(Lemma 3.3 of \cite{Sheridan-CY})
 The map $\rho$ defines a Lagrangian subbundle $F$ in the boundary of the symplectic bundle $E = u^*TX$
	and  
	$$
		\mu(E, F) \mapsto [\rho]
	$$
	in the right-exact sequence \eqref{e:ZtoGX0}.   
\end{lemma}

Recall that an \emph{anchored} Lagrangian brane is a Lagrangian submanifold $L\subseteq X$ which is equipped with a Pin structure and a lift $L^\#\subseteq \widetilde{\gra}(X)$, where $\widetilde{\gra}(X)\rightarrow \gra(X)$ is the covering map corresponding to the commutator subgroup of $\pi_1(\gra(X))$. In particular, $Y = H_1(\gra(X))$ acts on $\widetilde{\gra}(X)$ by deck transformations, and this action is transitive on the fibers of the covering map. 

Let $L_0$ and $L_1$ be anchored Lagrangian branes, and let $y:[0,1]\rightarrow X$ be a non-degenerate Hamiltonian chord from $L_0$ to $L_1$. Let $g \in \ker(Y\to \mathbb{Z}_2)$ be a class such that $y$ lifts to a path $y^\#$ from $L_0^\#$ to $g\cdot L_1^\#$. 
We can associate with $y^\#$ a Cauchy-Riemann operator $D_{y^\#}$ on the upper-half plane; let $k$ be its index.  
Then, the degree of $y$ is defined to be $\deg(y) \defeq k-g$, and it is independent of the path $y^\#$. 
We define $\det(D_{y^\#})$ to be the determinant line of the Fredholm operator $D_{y^\#}$, placed in degree $\deg(y)$. 
We define $\mathrm{Pin}_{y^\#}$ to be the $\mathbb{G}$-graded $\Z_2$-torsor whose elements are the two isomorphism classes of Pin structures on the pullback of the tautological bundle by $y^\#$, equipped with identifications with the Pin structures on $L_0^\#$ and $L_1^\#$ at the ends; it is placed in degree $0$. 
We define the orientation line $o_{y^\#} \defeq \det(D_{y^\#}) \otimes \mathrm{Pin}_{y^\#}$. 

We claim that $o_{y^\#}$ is independent of the choice of $y^\#$ up to canonical isomorphism. 
Indeed, by a computation of first Stiefel--Whitney class, the principal $\Z_2$-bundle whose fibres over $y^\#$ is $o_{y^\#}$ is trivial; so the fibres in each connected component can be canonically identified. 

It remains to identify fibres $o_{y_1^\#} \cong o_{y_2^\#}$ whose Maslov indices differ by an even integer $2i$. 
This is done by a standard gluing argument, see e.g. \cite[Proposition 8.1.4]{FOOO} or \cite[Proposition 1.4.10]{Abouzaid13}. 
Observe that $D_{y^\#_2}$ is isomorphic to the gluing $D_{y^\#_1} \# D_{\mathbb{CP}^1,i}$ where $D_{\mathbb{CP}^1,i}$ is the Cauchy--Riemann operator associated to a complex vector bundle over $\mathbb{CP}^1$ with first Chern class $i$, which gets glued to the upper half plane at an interior node. 
The resulting isomorphism
$$\det(D_{y^\#_2}) \otimes \det(\C^n) \cong \det(D_{y^\#_1}) \otimes \det(D_{\mathbb{CP}^1,i}),
$$
together with the canonical complex orientations of $\det(\C^n)$ and $\det(D_{\mathbb{CP}^1,i})$, gives the desired isomorphism $\det(D_{y^\#_2}) \cong \det(D_{y^\#_1})$. 
Associativity of gluing shows that these identifications are compatible (i.e., the composition of isomorphisms $o_{y_1^\#} \cong o_{y_2^\#} \cong o_{y_3^\#}$ is the chosen isomorphism $o_{y_1^\#} \cong o_{y_3^\#}$.)

Thus, we may unambiguously write $o_y$ for this canonically identified family of lines $o_{y^\#}$.

\subsection{Canonical index}\label{subsec:canind}
We now describe a similar grading scheme for $1$-periodic Hamiltonian orbits. Consider a smooth path of symplectic matrices 
\begin{equation}\label{e:PathSympMat}
	\Phi : [0,1] \to Sp(\mathbb{R}^{2n}) \quad\mbox{such that}\quad \Phi(0) = \id \mbox{ and }\ker (\Phi(1) - \id) = 0.
\end{equation}
Up to homotopy with fixed endpoints, we may assume that
\begin{equation}
 A(t) \defeq (\partial_t \Phi(t)) \cdot \Phi(t)^{-1} \in sp(\mathbb{R}^{2n})	
\end{equation}
is $1$-periodic. Let $B: Z_- \to sp(\mathbb{R}^{2n})$ be a smooth map on the negative half-cylinder $Z_- = (-\infty, 0] \times S^1$ such that
\begin{equation}
 B(s, t) = A(t) \mbox{ for } s \ll 0 \quad\mbox{and}\quad B(s, t) = 0 \mbox{ near } s = 0.	
\end{equation}
For each smooth loop $F: S^1 \to \gra(\mathbb{R}^{2n})$ of linear Lagrangians, we can define a real linear Cauchy-Riemann operator
\begin{align}\label{e:ConOp}
	D_{\Phi,F}: W^{1,q}(Z_-, \mathbb{R}^{2n}, F) &\to L^q(\hom^{0,1}(TZ_-, \mathbb{R}^{2n}))\\ 
	D_{\Phi, F}(\zeta) &= (d\zeta - B\cdot\zeta \otimes dt)_{J_0}^{0,1},\notag
\end{align}
where $W^{1,q}(Z_-, \mathbb{R}^{2n},F) = \{\zeta \in W^{1,q}(Z_-, \mathbb{R}^{2n}) : \zeta(0,t) \in F(t)\}$ and $J_0$ is the standard complex structure on $\mathbb{R}^{2n}$. When $q > 2$, the operator $D_{\Phi, F}$ is Fredholm, see \cite[Appendix C]{mcduff-salamon-Jholomorphiccurves} for a proof.

Let $H: S^1 \times X \to \mathbb{R}$ be a Hamiltonian. With each pair $(x, \rho)$ consisting of a non-degenerate $1$-periodic orbit $x : S^1 \to X$ of $X_H$ and a lift $\rho : S^1 \to \gra(X)$ of $x$, we associate a real linear Cauchy--Riemann operator as before
\begin{equation}\label{e:CanonOp}
	D_{x, \rho}: W^{1,q}(Z_-, v_x^* TX,\rho) \to L^q(\hom^{0,1}(TZ_-, v_x^* TX)),
\end{equation}
where $v_x: Z_- \to X$ is given by $v_x(s,t) = x(t)$.
In a unitary trivialization $\Psi : x^* TX \cong S^1 \times \mathbb{R}^{2n}$, this operator has the form
$D_{\Phi, F}$ as in \eqref{e:ConOp} where $\Phi(t) = \Psi(t) d\phi_H^t \Psi(0)^{-1}$
and $F(t) = \Psi(t) \rho(t)$.  Hence, the operator $D_{x, \rho}$ is Fredholm and, when $x$ is fixed, its index depends only on the homotopy class of the Lagrangian lift $\rho$.
\begin{definition}\label{Canonicalops}
For each non-degenerate $1$-periodic Hamiltonian orbit $x$, let $[\rho_x] \in H_1(\gra(X))$ be the homology class of the (unique up to homotopy) lift $\rho_x: S^1 \to \gra(X)$ of $x$ for which $\mbox{index}(D_{x, \rho_x}) = 0$. 
Define the \emph{degree} of $x$ to be
\begin{equation}
	\deg(x) \defeq n - [\rho_x] \in H_1(\gra(X)),
\end{equation}
where $n$ represents the image of $n \in \mathbb{Z} \to H_1(\gra(X))$ in \eqref{e:ZtoGX0}.   
\end{definition}

There is an analogous operator $D^{\vee}_{x, \bar{\rho}_x}$ defined on the positive half-cylinder $Z_+ = [0, \infty) \times S^1$
with respect to the map $v^\vee_x: Z_+ \to X$, where $\bar{\rho}_x$ is $\rho_x$ with the reverse orientation.

\subsection{Index formulae}

Let $\Sigma = \mathbb{P}^1 \setminus \mathbf{z}$ be a genus $0$ open Riemann surface, where $\mathbf{z} = (z_0, \dots, z_d)$ are distinct points on $\mathbb{P}^1$. Fix a collection $(\epsilon_0, \dots, \epsilon_d)$ of cylindrical ends for $\Sigma$ such that $\epsilon_0$ is a negative end and $\epsilon_1, \dots, \epsilon_d$ are positive ends.

Suppose we have a domain-dependent Hamiltonian $K \in \Omega^1(\Sigma,C^{\infty}(X))$ and a domain-dependent $\omega$-compatible almost complex structure $J$ such that, on the cylindrical ends,
\begin{equation*}
 \epsilon_k^* K = H_k \otimes dt \quad\mbox{and}\quad \epsilon_k^* J = J_{k}   
\end{equation*}
are only time-dependent. Pick a non-degenerate $1$-periodic orbit $x_k$ for each Hamiltonian $H_k: S^1 \times X \to \mathbb{R}$ and consider solutions $u: \Sigma\to X$ of the perturbed pseudo-holomorphic curve equation
\begin{equation}\label{pseudoholomorphic-equation}
\begin{cases}
 (du - Y_K)_{J}^{0,1} &=\ 0 \\
 \lim\limits_{s\to -\infty} u(\epsilon_0(s,t)) &=\ x_0 \\
 \lim\limits_{s\to \infty} u(\epsilon_i(s,t)) &=\ x_i, \ \text{for} \ i=1,\dots,d. 
\end{cases} 
\end{equation}
Here, $Y_K \in \hom(T\Sigma, TX)$ maps $\xi\in T\Sigma$ to $X_{K(\xi)}$; the Hamiltonian vector field associated with $K(\xi)$.
Following standard pseudo-holomorphic curve theory, the linearization of equation \eqref{pseudoholomorphic-equation} at a solution $u$ is a real linear Cauchy-Riemann operator
\begin{equation}
 D_{u}: W^{1,q}(\Sigma, u^*TX) \to L^q(\hom^{0,1}(T\Sigma, u^* TX)).
\end{equation}
The operator $D_u$ is Fredholm and its index represents the virtual dimension of the moduli space of solutions to \eqref{pseudoholomorphic-equation}  near $u$.
\begin{lemma}\label{index lemma}
Under the map $\mathbb{Z} \to H_1(\gra(X))$ from \eqref{e:ZtoGX0},
	$$
		\emph{\text{index}}(D_u) \mapsto \deg(x_0) - \sum_{i=1}^d \deg(x_i).
	$$   
\end{lemma}
\begin{proof}
By gluing the canonical operators (see Definition \ref{Canonicalops})
$D^\vee_{x_0,\rho_{x_0}}$ and $(D_{x_i,\rho_{x_i}})_{i\geq 1}$ to $D_u$
using the cylindrical ends, we obtain a new Fredholm operator
\begin{equation}
	D = D^\vee_{x_0,\bar{\rho}_{x_0}} \# D_{u} \# D_{x_1,\rho_{x_1}} \# \cdots \# D_{x_d,\rho_{x_d}}
\end{equation}
defined over a genus $0$ Riemann surface with boundary $S$. The surface $S$ is given by removing $d+1$ disjoint disks from $\mathbb{P}^1$.  Explicitly, $D$ is a Fredholm operator
\begin{equation}
	D: W^{1,q}(S, v^* TX,\boldsymbol\rho) \to L^q(\hom^{0,1}(TS, v^* TX)),
\end{equation}
where $v$ is the result of gluing $u$ to $v_{x_0}^\vee$ and $(v_{x_i})_{i\geq 1}$, and the Lagrangian boundary condition is determined by $\boldsymbol\rho = (\bar{\rho}_{x_0}, \rho_{x_1}, \dots, \rho_{x_d})$.
Using the gluing formula for Fredholm operators, we see that
$$
	\ind(D_u)= \ind(D),
$$
because the canonical operators all have index $0$. Therefore, using the Riemann-Roch theorem (see \cite[C.1.10]{mcduff-salamon-Jholomorphiccurves}) and Lemma \ref{l:IndexToLoop}, we have
\begin{align*}
	 \ind(D) &= n\, \chi(S) + \mu(v^* TX,\boldsymbol\rho) \\
	&\mapsto n (1 - d) + [\boldsymbol \rho] \\
	&= n (1 - d) - [\rho_{x_0}] + [\rho_{x_1}] + \dots + [\rho_{x_d}]\\ 
	&= \deg(x_0) - \deg(x_1) - \dotsi - \deg(x_d).
\end{align*}
\end{proof}
There is an analogous index formula when $\Sigma$ has boundary. Let $(L_i)_{i=1}^k$ be a collection of Lagrangians, and consider a disc $(\mathbb{D}, \zeta_0,\dots,\zeta_k,z_1,\dots,z_d) \in \Rms^\disc_{d,k}$ which comes with asymptotic ends for its marked points. Set $\Sigma = \mathbb{D}\backslash \{\zeta_0,\dots,\zeta_k,z_1,\dots,z_d\}$ and consider solutions $u:\Sigma \rightarrow X$ of the equation
\begin{equation}\label{popsicle equation'}
    \begin{cases}
        & (du-Y_r)_{J_r}^{0,1} = 0,\\
        & u((\partial\Sigma_r)_i) \subseteq L_i\\
		& \lim_{s\rightarrow \pm\infty} \epsilon_{\zeta_i}^*u(s,t) = y_{i}, \\
        & \lim_{s\rightarrow \pm\infty} \epsilon_{z_i}^*u(s,t) = x_{i}.
    \end{cases}
\end{equation}
Here, $(\partial\Sigma)_i$ is the boundary component which originates from $\zeta_i$ with respect to the boundary orientation, the $y_i$'s are non-degenerate Hamiltonian chords, and the $x_i$'s are non-degenerate $1$-periodic orbits. 

The linearization of equation \eqref{popsicle equation'} at a solution $u$ is a Cauchy-Riemann operator $D_u$ which is Fredholm. Its index, under the map $\mathbb{Z} \rightarrow H_1(\gra(X))$, maps to
\begin{equation}
	\text{index}(D_u) = \deg(y_0) - \sum_{i=1}^{k} \deg(y_i) -\sum_{i=1}^{d} \deg(x_i).
\end{equation}
The proof is similar to Lemma \ref{index lemma}.

\subsection{Orientation operators}\label{s:orline}
Let $x: S^1 \to X$ be a non-degenerate Hamiltonian orbit. We choose a trivialization of the complex vector bundle $x^* TX$ over $S^1$,
\begin{equation}\label{trivialization of orbit}
 x^* TX \cong \mathbb{C}^n \times \mathbb{D}|_{\partial \mathbb{D}}.   
\end{equation}
In \cite[\S 1.4]{Abouzaid13}, Abouzaid defines a Cauchy-Riemann operator $D_x$ over $\mathbb{C}$, regarded as a Riemann surface with a single output at infinity. This operator maybe thought of as the gluing of the canonical operator $D_{x,\rho_x}$ from Definition \ref{Canonicalops} and $D_{\mathbb{D},\rho_x}$; the operator on the disk $\mathbb{D}$ with boundary condition along the Lagrangian loop $\rho_x$. In particular, using standard gluing theorems (see \cite[(11c)]{PL-theory}), together with the Riemann-Roch theorem (see \cite[Lemma 11.7]{PL-theory}) we have that
\begin{align}
 \ind(D_x) 
	&= \ind(D_{\mathbb{D},\rho_x}) \\
	&= n + \mu(\rho_x) \notag\\
	&\equiv n + w_1(\rho_x) \mod 2.\notag
\end{align}
In particular, $\ind(D_x) = \deg(x) \mod 2$, under the map $H_1(\gra(X)) \rightarrow \mathbb{Z}_2$ described in Definition \ref{grading datum}.

The orientation line associated with $x$ is the $\mathbb{G}$-graded line 
\begin{equation}
	o_x \defeq \det(D_x),
\end{equation}
placed in degree $\deg(x)$.
Different choices of trivialization \eqref{trivialization of orbit} give rise to canonically isomorphic lines by the same argument as before: gluing Cauchy--Riemann operators over $\mathbb{CP}^1$, see \cite[Proposition 1.4.10]{Abouzaid13}.

\section{Floer theory in Liouville domains}\label{sec:Floer}

A Liouville domain is a tuple $(X,\omega,\theta)$ where $X$ is a smooth manifold with boundary, $\omega = d\theta$ is a symplectic form, and the Liouville vector field, defined by
\begin{equation}
    \iota_Z\omega = \theta,
\end{equation}
points outwards along the boundary $\partial X$. A Liouville domain has a well-defined radial function $r:X\rightarrow \mathbb{R}_{\geq 0}$ such that $\log(r)$ is the time it takes a point $p\in X$ to reach $\partial X$ when $p$ is flown using the Liouville vector field $Z$. This radial function provides a collar neighborhood of the boundary
\begin{equation}\label{collar neighborhood}
    \phi^{\log(r)}_Z : (0,1]\times \partial X \rightarrow X,
\end{equation}
which in turn can be used to define the completion
\begin{equation}
    \widehat{X} = (X,\theta) \cup (\mathbb{R}_{>0}\times \partial X, r\theta_{\partial X})/ \sim,
\end{equation}
where $\sim$ is gluing over the collar (\ref{collar neighborhood}). We abuse notation and continue to use $\omega$, $\theta$ and $Z$ for the Liouville structure on $\widehat{X}$. It is useful to remember that in terms of the radial coordinate, $Z = r\partial_r$.

A time-dependent Hamiltonian $H:S^1\times X \rightarrow \mathbb{R}$ has an associated vector field $X_H\in C^{\infty}(S^1\times X,TX)$ defined by the equation 
\begin{equation}
    \iota_{X_{H}}\omega = -d_XH.
\end{equation}
Its set of $1$-periodic orbits is
\begin{equation}
    \mathcal{X}_{H} \defeq \{ x:S^1\rightarrow X \ | \ \dot{x}(t) = X_{H}(x(t)) \}.
\end{equation}
It is in one to one correspondence with the intersection of $\text{Graph}(\phi_1^{X_H})\subseteq X\times X$ with the diagonal, where $\phi_1^{X_H}:X\rightarrow X$ is the time $1$-flow of $X_H$. When this intersection is transverse, we say that $H$ is \emph{non-degenerate}.

We will require a special type of Hamiltonians which we now describe. Fix $r_0 \in (0,1)$ and $\epsilon\in (0,1-r_0)$, and let $h:\mathbb{R}_{\geq 0}\rightarrow \mathbb{R}_{\geq 0}$ be a non-decreasing smooth function such that
\begin{equation}
    \begin{cases}
        & h(r) = 0 \quad \text{if} \ r\leq r_0,\\
        & h''(r) = 0 \quad \text{if} \ r>1-\epsilon,\\
        & n\cdot h'(1) \notin \text{Spec}(R_{\theta}),\ \text{for all}\ n\in \mathbb{N}.
    \end{cases}
\end{equation}
Here, $\text{Spec}(R_{\theta})\subseteq \mathbb{R}$ is the set of periods of the Reeb flow on $\partial X$. 
\begin{definition}
A \emph{basic sequence} of Hamiltonians is a non-decreasing sequence $\bH = (H_n)_{n=1}^{\infty}$ of non-degenerate Hamiltonians on $X$ of the form $H_n = n\cdot h(r) + K_n$, such that
\begin{itemize}
    \item[(i)] The sequence $K_n$ is non-positive and compactly supported in the region $\{r \leq 1-\epsilon\}$.
    \item[(ii)] There is a sequence of autonomous Hamiltonians  $C_n:X\rightarrow \mathbb{R}$ for which $K_{n-1} \leq C_n \leq K_n$.
    \item[(iii)] The orbits of the $H_n$ are pairwise disjoint.
\end{itemize}
By convention, $H_0 = 0$.
\end{definition}

\begin{figure}
    \centering
    \begin{tikzpicture}[x=0.75pt,y=0.75pt,yscale=-1,xscale=1]
        
        \draw  (143,331.8) -- (560,331.8)(184.7,87) -- (184.7,359) (553,326.8) -- (560,331.8) -- (553,336.8) (179.7,94) -- (184.7,87) -- (189.7,94)  ;
        \draw   (340,332) -- (348,332)(344,327) -- (344,337) ;
        \draw   (429,332) -- (437,332)(433,327) -- (433,337) ;
        \draw   (512,332) -- (520,332)(516,327) -- (516,337) ;
        \draw    (184.7,331.8) -- (363,332) ;
        \draw    (363,332) .. controls (394,334) and (415,309) .. (430,293) ;
        \draw    (430,293) -- (516,209) ;
        
        \draw (338,340) node [anchor=north west][inner sep=0.75pt]   [align=left] {$r_0$};
        \draw (415,337) node [anchor=north west][inner sep=0.75pt]   [align=left] {$1-\epsilon$};
        \draw (510,340) node [anchor=north west][inner sep=0.75pt]   [align=left] {$1$};
        \draw (566,328) node [anchor=north west][inner sep=0.75pt]   [align=left] {$r$};
        \draw (172,64) node [anchor=north west][inner sep=0.75pt]   [align=left] {$h(r)$};
        \end{tikzpicture}
        \caption{}
        \label{Fig: h(r)}

\end{figure}

\subsection{Hamiltonian Floer theory}\label{Hamiltonian Floer theory}

Let $\bH$ be a basic sequence. We denote by $\mathcal{X}_n$ the set of $1$-periodic orbits of $H_n$. Each $x\in \mathcal{X}_{n}$ has an associated $\mathbb{G}$-graded line $o_x$, and its normalization $\abs{o_x}$. 
Floer's complex associated with $H_n$ is the $\mathbb{G}$-graded vector space
\begin{equation}\label{Floer complex}
    CF^*(X,H_n)  = \bigoplus_{x\in \mathcal{X}_{n}} \abs{o_x}.
\end{equation}
It is equipped with a differential which counts maps $u:\mathbb{R}\times S^1 \rightarrow X$ satisfying the pseudo-holomorphic equation
\begin{equation}\label{Floer trajectory}
    \begin{cases}
       & \partial_s u + J_t (\partial_t u - X_{H_n}(u)) = 0\\
       & E(u)\defeq \int \abs{\partial_s u}^2 < +\infty,
    \end{cases}
\end{equation}
where $J_t$ is an appropriately chosen almost complex structure (see \S \ref{subsubsec: almost complex structures}).
For each \emph{distinct} pair $x_-,x_+ \in \mathcal{X}_{n}$, define the moduli space
\begin{equation}\label{Floer trajectories}
    \Mms^\al(x_-,x_+) = \left\{ u\  \text{satisfies (\ref{Floer trajectory}) and}\ \lim_{s\rightarrow \pm\infty} u(s,t) = x_{\pm} \right\}/\mathbb{R},
\end{equation}
where the action of $\mathbb{R}$ is by translation in the $s$-direction, $r\cdot u(s,t) = u(s+r,t)$. Floer's differential on (\ref{Floer complex}) is
\begin{equation}\label{Floer-differential}
    dx_+ = \sum_{x_-\neq x_+}\# \Mms^\al(x_-,x_+)\cdot x_-.
\end{equation}
Here, $\# \Mms^\al(x_-,x_+)$ is notation for the signed count of isolated points in this moduli space.

Given any two terms $H^-$ and $H^+$ of the basic sequence $\bH$ such that $H^- \geq H^+$, we can also construct a continuation chain map
\begin{equation}\label{continuation map}
    c: CF^*(X,H^+) \rightarrow CF^*(X,H^-).
\end{equation}
Let $H_s = n(s)\cdot h(r) + K_{s}$ be a decreasing homotopy from $H^-$ to $H^+$, so that
\begin{equation}\label{decreasing H}
  \partial_s H_s \leq 0 \quad\text{and}\quad  H_{\pm s} = H^{\pm} \quad \text{for}\  s \gg 0.
\end{equation}

Continuation elements are maps $u:\mathbb{R}\times S^1 \rightarrow X$ which satisfy
\begin{equation}\label{continuation equation}
    \begin{cases}
       & \partial_s u + J_{s,t} (\partial_t u - X_{H_s}(u)) = 0\\
       & E(u)\defeq \int \abs{\partial_s u}^2 < +\infty,
    \end{cases}
\end{equation}
where $J_{s,t}$ is an appropriately chosen family of almost complex structures. For each (not necessarily distinct) pair $x_{\pm}\in\mathcal{X}_{H^{\pm}}$, define the moduli space
\begin{equation}
    \Mms^\al_{c}(x_-,x_+) = \left\{ u\  \text{satisfies (\ref{continuation equation}) and}\ \lim_{s\rightarrow \pm\infty} u(s,t) = x_{\pm} \right\}.
\end{equation}
Similar to Floer's differential (\ref{Floer-differential}), the continuation map is
\begin{equation}
    c(x_+) = \sum \# \Mms^\al_{c}(x_-,x_+)\cdot x_-.
\end{equation}

This construction produces a sequence of continuation maps
\begin{equation}
    CF^*(X,H_n) \rightarrow CF^*(X,H_{n+1}).
\end{equation}
Their homotopy direct limit (see \cite[\S 3.7]{open-string-analogue}) is the complex 
\begin{equation}\label{telescope complex}
    SC^*(X,\mathbf{H}) = \bigoplus_{n=1}^{\infty}CF^*(X,H_n)[t],
\end{equation}
where $t$ is a formal variable of degree $-1$ such that $t^2=0$. It is equipped with the differential
\begin{equation}\label{eq:dSC}
    d^{SC}(x+tx') = dx - tdx' + c(x') - x'.
\end{equation}
The cohomology of this complex is called symplectic cohomology and is denoted $SH^*(X)$. It only depends on the deformation type of the Liouville domain $X$.

More generally, let $r = [(C,z_0,\dots,z_d;\theta_0),\psi]\in \Rms^{\al}_{d,\bp,\bn}$ be a (smooth) $\bp$-flavoured genus $0$ curve with $d$-marked points and aligned framings, carrying a sprinkle $\psi$, and weights $\bn$. There is a generalized Floer equation for maps $u:\Sigma_r \rightarrow X$ in the complement $\Sigma_r = C\backslash\{z_0,\dots,z_d\}$:
\begin{equation}\label{lollipop equation}
    \begin{cases}
        & (du - Y_r)_{J_r}^{0,1} = 0 \\
        & E(u) \defeq \int \abs{du-Y_r}^2 < +\infty.
    \end{cases} 
\end{equation}
In equation (\ref{lollipop equation}), $J_r = (J_z)_{z\in\Sigma_r}$ is an appropriate domain-dependent almost complex structure, $Y_r \in \Omega^1(\Sigma_r, C^{\infty}(TX))$ is a $1$-form on $\Sigma_r$ with values in \emph{Hamiltonian} vector fields on $X$, and 
\begin{equation}
    (du - Y_r)_{J_r}^{0,1} \defeq  du - Y_r + J_r\circ (du - Y_r)\circ j_{\Sigma_r}.
\end{equation}
The perturbation datum $Y_r$ is $\omega$-dual to a $1$-form of Hamiltonian vector fields, which we can lift to a $1$-form of Hamiltonian functions $\mathcal{K}_r \in \Omega^1(\Sigma_r, C^{\infty}(X))$ which satisfies
\begin{equation}\label{end constraints}
    \epsilon_i^* \mathcal{K}_r = H_{n_i}dt \quad \text{for}\ i=0,1,\dots,d\ \ \text{and} \ \pm s\gg 0.
\end{equation}

For example, when $\Sigma_r$ is a cylinder with one flavour (i.e. $F=\{1\}$), equation (\ref{lollipop equation}) is the same as the equation (\ref{continuation equation}) for continuation elements, with 
\begin{equation}\label{Hamiltonian for 1 sprinkle}
Y_r = X_{H_{s,t}}\otimes dt.
\end{equation}

As before, let $\bx = (x_0,\dots,x_d)$ be a collection of $1$-periodic orbits $x_0\in \mathcal{X}_{n_0},\dots, x_d\in \mathcal{X}_{n_d}$. Then, we have a moduli space of solutions
\begin{equation}
    \Mms^{\al}_r(\bx) = \{ u\ \text{satisfies (\ref{lollipop equation}) and}\ \lim_{s\rightarrow \pm\infty} \epsilon_i^*u(s,t) = x_{i} \ \text{for}\ i=0,\dots, d\}.
\end{equation}
The family version allows for $r = [(C,z_0,\dots,z_d;\theta_0),\psi]$ to vary in its moduli space,
\begin{equation}
    \Mms^{\al}_{d,\bp,\bn}(\bx) = \{ (r,u) \ \vert \ r\in \Rms^{\al}_{d,\bp,\bn} \ \text{and} \ u \in \Mms^{\al}_r(\bx) \}.
\end{equation}

\subsection{Transversality and compactness}\label{subsec:Transversality and compactness}

Having well-behaved moduli spaces of pseudo-holomorphic curves hinges on making appropriate choices of perturbation data $(J_r,Y_r)$. We explain how to make such choices for $r \in \Rms^{\al}_{d,\bp,\bn}$.
\subsubsection{Hamiltonian perturbations} The perturbation datum $Y_r$ is the $\omega$-dual of $\mathcal{K}_r = H_r\otimes \gamma_r \in \Omega^{1}(\Sigma_r,C^{\infty}(X))$, where $H_r$ is a domain-dependent Hamiltonian and $\gamma_r\in \Omega^1(\Sigma_r)$ is a $1$-form.
We require our perturbation to satisfy
\begin{equation}
    \epsilon_i^* \mathcal{K}_r = H_{n_i}dt \quad \text{for}\ i=0,1,\dots,d\ \ \text{and} \ \pm s\gg 0.
\end{equation}

This is achieved by choosing $\gamma_r$ such that
\begin{equation} \label{conditions on gamma}
    \begin{cases}
    \epsilon_i^*\gamma_r = n_idt \quad &\text{for}\ i=0,1,\dots,d\ \ \text{when} \ \pm s\gg 0,\\
    d\gamma_r \leq 0 \quad &\text{i.e. }\gamma_r\ \text{is sub-closed}.
    \end{cases}
\end{equation}
The Hamiltonian term is chosen to have the form $H_r = h(r) + F_r$, where $F_r$ is a domain-dependent Hamiltonian with compact support in $\{r \leq 1-\epsilon \}$ such that
\begin{equation}
    \epsilon_i^* F_r = K_{n_i}/n_i \quad \text{for}\ i=0,1,\dots,d\ \ \text{and} \ \pm s\gg 0.
\end{equation}

The curvature of $\mathcal{K}_r$ is defined as
\begin{equation}\label{curvature term}
    R(\mathcal{K}_r)(y) \defeq d_{\Sigma_r}(H_r(-,y)\gamma_r), \quad \text{for}\ y\in X.
\end{equation}
Note that compactness of our moduli spaces requires a curvature estimate
\begin{equation}
    \int_{\Sigma_r} R(\mathcal{K}_r)(-,y) \leq C_{d,\bp,\bn}.
\end{equation}

The construction of $\gamma_r$ is easier. The idea is that when $r\in \Rms^{\al}_{d,\bp,\bn}$ is fixed, the space of $1$-forms satisfying \eqref{conditions on gamma} is convex and nonempty, hence contractible. Moreover, the constraints \eqref{conditions on gamma} are compatible with the gluing operation on the moduli spaces $\Rms^{\al}_{d,\bp,\bn}$. Hence, it suffices to construct $\gamma_r$ in two cases.
\begin{itemize}
    \item[-] Case 1: When $\Sigma_r$ is a cylinder and $n_0=1+n_1$. Simply use $\gamma_r = \rho_n(s)dt$ where $\rho:\mathbb{R}\rightarrow \mathbb{R}$ is a non-increasing function such that
    \begin{equation}
        \rho_n(s) =
        \begin{cases}
            n_1 \quad &\text{for}\quad s \gg 0,\\
            n_0 \quad &\text{for}\quad s \ll 0.
        \end{cases}
    \end{equation}

    \item[-] Case 2: When $\Sigma_r$ is a pair of pants and $n_0 = n_1 + n_2$. Start with any $1$-form $\alpha_1$ such that $\epsilon_i^*\alpha_1 = n_i dt$. Then $d\alpha_1\in \Omega_c^2(\Sigma_r)$ is a closed and compactly supported $2$-form which integrates to $0$. Hence, there is a compactly supported $1$-form $\alpha_2$ such that $d\alpha_1 = d\alpha_2$. 
    Then $\gamma_r = \alpha_1-\alpha_2$ satisfies the conditions of \eqref{conditions on gamma}. 
\end{itemize}

Once we have constructed $\gamma_r$ in these base cases, we can inductively construct it on all the moduli spaces $\Rms^{\al}_{d,\bp,\bn}$ by gluing, see \cite[\S (9g)]{PL-theory} for an example of this inductive argument. See also \cite[\S 2.6]{open-string-analogue} for a similar construction for moduli spaces of discs.

The construction of the Hamiltonian term $H_r = h(r) + F_r$ is trickier. 
Recall that we have a sequence $C_n:X\rightarrow \mathbb{R}$ of autonomous Hamiltonians such that for all $n\geq 2$, 
\begin{equation}\label{increasing Kn}
    K_{n-1}\leq  C_{n} \leq K_n \leq 0.
\end{equation}
Define the constants 
\begin{equation}
    \widetilde{C}_n = \min_y \frac{C_{n}(y)}{n}.
\end{equation}

\begin{lemma}\label{Hamiltonian piece}
    For any $r\in \Rms^{\al}_{d,\bp,\bn}$, $\Sigma_r$ carries a Hamiltonian $F_r$ such that the perturbation datum $\mathcal{K}_r = (h(r)+F_r)\otimes\gamma_r$ satisfies
    \begin{equation}\label{curvature inequality}
     R(\mathcal{K}_r) \leq \widetilde{C}_{n_0} d\gamma_r.
    \end{equation}
\end{lemma}

\begin{proof}
    The case when $d=1$ is straightforward, so we assume $d\geq 2$. The inequality \eqref{increasing Kn} implies (by convex interpolation) that there is a smooth family of time-dependent Hamiltonians $\{K_{v} \ |\ v\in \mathbb{R}_{\geq 1} \}$ such that
    \begin{equation}
        \begin{cases}
         K_v = K_n \quad &\text{when} \ v = n \in \mathbb{N},\\
         K_{v+1/2} = C_{n+1} \quad &\text{when} \ v = n \in \mathbb{N},\\
         \partial_v K_v \geq 0.
        \end{cases}
    \end{equation}

    Consider a Riemann surface $\Sigma_r$, with $r\in \Rms^{\al}_{d,\bp,\bn}$ such that $d\geq 2$. We construct the Hamiltonian $F_r\in C^{\infty}(\Sigma_r\times X)$ in pieces. On the cylindrical end $\epsilon_i$, we construct
    \begin{equation}
        (\epsilon_i^* F_r)(s,t) = \frac{K_{f_i(s)}}{f_i(s)},
    \end{equation}
    where $f_i$ is a decreasing function such that
    \begin{equation}
        f_i(s) = 
        \begin{cases*}
         n_i &\text{when}\ $\pm s\gg 0$\\
         n_0-1/2 &\text{when}\ $\pm s$ \text{is near} \ 0.
        \end{cases*}
    \end{equation}
    (Note that $n_0 \ge n_1+n_2 \ge 2$, so $n_0 - 1/2 \ge 1$, hence $K_{f_i(s)}$ is well-defined.)
    On the complement of the cylindrical ends, we construct
    \begin{equation}
        F_r = \frac{C_{n_0}}{n_0}
    \end{equation}
    which is domain-independent. 
    Setting $\mathcal{K}_r = (h(r)+F_r)\otimes\gamma_r$, it is not difficult to see that
    \begin{itemize}
        \item[-] Outside the cylindrical ends,
        \begin{equation}
            R(\mathcal{K}_r) \leq (h(r)+F_r)d\gamma_r \leq \frac{C_{n_0}}{n_0} d\gamma_r.
        \end{equation}
        \item[-] On the cylindrical end $\epsilon_i$,
        \begin{equation}
            \epsilon_i^*R(\mathcal{K}_r) = f_i'(s)\left(\frac{\partial_v K_{f_i(s)}}{f_i(s)} - \frac{K_{f_i(s)}}{f_i(s)^2}\right)ds\wedge dt\leq 0
        \end{equation}
    \end{itemize}
    Overall, $R(\mathcal{K}_r)\leq \widetilde{C}_{n_0}d\gamma_r$.
\end{proof}
In order to make a consistent universal choice of $\mathcal{K}_r$, we define
    \begin{equation}
        C_{\bn} =   
            \begin{cases}
                 0 \quad &\text{if}\ d = 1\\
                 \min_{j} \widetilde{C}_{n_j} \quad &\text{if}\ d\geq 2,
            \end{cases}
    \end{equation}
These new constants have the advantage that when $r \in \Rms^{\al}_{d,\bp,\bn}$ breaks into a pair of aligned framed curves $r_{\pm} \in \Rms^{\al}_{d_{\pm},\bp_{\pm},\bn_{\pm}}$, we have
\begin{equation}\label{min-property}
    C_{\bn} = \min(C_{\bn_-},C_{\bn_+}).
\end{equation}
That is because the constants $\widetilde{C}_{n}$ are increasing in $n$, so that $\widetilde{C}_{\bn_{-,0}}\geq \widetilde{C}_{\bn_{-,1}}$.

\begin{corollary} \label{Hamiltonian perturbations}
There is a consistent universal choice of Hamiltonian perturbations $\mathcal{K}_r = H_r\otimes \gamma_r$ for every $r\in \Rms^{\al}_{d,\bp,\bn}$ which is $\Sym(d,\bp)$-invariant and such that
\begin{equation}
    R(\mathcal{K}_r) \leq C_{\bn}d\gamma_r.
\end{equation}
\end{corollary}

\begin{proof}
For each $r\in \Rms^{\al}_{d,\bp,\bn}$, the collection of Hamiltonian perturbations of the form $\mathcal{K}_r  = (h(r)+F_r)\otimes \gamma_r$ such that
$R(\mathcal{K}_r) \leq C_{\bn}d\gamma_r$
is non-empty (by Lemma \ref{Hamiltonian piece}) and convex, hence contractible. Moreover, a choice of perturbations on $\partial\Rms^{\al}_{d,\bp,\bn}$ can be extended to a neighborhood of the boundary by gluing. This is due to the property \eqref{min-property}, together with the fact that the sub-closed $1$-forms $\gamma_r$ are already compatible with gluing. Therefore, following the inductive process of \cite[\S (9g)]{PL-theory}, we can make a consistent universal choice of perturbations as desired.

Finally, to ensure $\Sym(d,\bp)$-invariance, one can take the average over the orbits of this symmetry group.
\end{proof}

\subsubsection{Almost complex structures} \label{subsubsec: almost complex structures}
An almost complex structure $J$ on $X$ is said to be of \emph{contact type} if it is $\omega$-compatible and if on the region $\{r\geq 1-\epsilon \}\subseteq X$, we have
\begin{equation}
    dr = \theta\circ J.
\end{equation}
For each weight $n\in \mathbb{N}$, we fix a time-dependent $J_{n}$ which is of contact type and such that all Floer trajectories \eqref{Floer trajectories} of the Hamiltonian $H_{n}$ are Fredholm regular. 

Given $r\in \Rms^{\al}_{d,\bp,\bn}$, we want to choose a domain-dependent almost complex structure $J_r$ on $\Sigma_r$ which is of contact-type and such that
\begin{equation}
    \epsilon_{k}^* J_r \longrightarrow J_{n_k} \quad \text{as}\ \pm s \to \infty,
\end{equation}
asymptotically faster than any $\exp(-C\abs{S})$. 
\begin{lemma}\label{lem:transversality}
    Suppose the basic sequence $\bH$ is generic. Then, there exists a consistent universal choice of contact-type almost complex structures $J_r$ which is invariant under the action of $\Sym(d,\bp)$ such that the moduli spaces $\Mms^{\al}_{d,\bp,\bn}(x_0,\dots,x_d)$ are all Fredholm regular.
\end{lemma}

\begin{proof}
   This transversality argument is detailed in \cite[\S 8.3, 8.4]{open-string-analogue}. In our case, we need an alternative argument to exclude \emph{trivial solutions}, i.e. maps $u:\Sigma_r \rightarrow X$ which satisfy $du = Y_r$. Indeed, such maps satisfy the pseudo-holomorphic equation $(du-Y_r)^{0,1}_{J_r} = 0$ for \emph{any} $J_r$ and hence can't be made Fredholm regular by perturbations of the almost complex structure.
   
   Suppose now that $u$ is a solution of the differential equation
   \begin{equation}
    du = X_{r} \otimes \gamma_r.
   \end{equation}
    We have a solution whenever $d=1$, $n_0=n_1$, and $x_0=x_1$; but this is not one of the moduli spaces $\Mms^{\al}_{d,\bp,\bn}(x_0,\dots,x_d)$ that we care about. 
    We claim that there are no other solutions. 
    Indeed, in any other case, $x_0$ is distinct from $x_1,\ldots,x_d$. 
    Thus, as the orbits are pairwise disjoint by assumption, there exists a smooth function $f:X \to \R$ such that $f|_{\text{image}(x_0)} < -\epsilon$ and $f|_{\text{image}(x_i)} > \epsilon$ for $i=1,\ldots,d$. 
    The smooth function $f \circ u: \Sigma_r \to \R$ admits a regular level set $C=(f \circ u)^{-1}(c)$ for some $c \in (-\epsilon,\epsilon)$, by Sard's theorem. 
    If $v$ is a tangent vector transverse to $C$, then we have
    $$X_r(f) \cdot \gamma_r(v) = d(f \circ u)(v) \neq 0,$$
    so $X_r(f) \neq 0$. 
    If $w$ is a tangent vector pointing along $C$, then
    $$X_r(f) \cdot \gamma_r(w) = d(f \circ u)(w) = 0,$$
    so $\gamma_r(w) = 0$. 
    Now define $U = \{f \circ u \ge c\} \subset \Sigma_r$. 
    Note that the boundary of $U$ consists of the input orbits, together with $C$, and we have shown $\gamma_r|_C = 0$.
    Applying Stokes' theorem to $U$, we obtain
    $$\sum_{i=1}^d n_i = \int_{\partial U} \gamma_r = \int_U d\gamma_r \le 0, $$
    a contradiction.
\end{proof}

The energy of a solution $u$ to equation \eqref{lollipop equation} is given by
\begin{equation}
    E(u)\defeq \frac{1}{2}\int_{\Sigma_r} \norm{du-Y_r}^2_{J_r},
\end{equation}
where the square norm of a linear map $\nu: T\Sigma_r\rightarrow TX$ is the $2$-form
\begin{equation}
    \norm{\nu} ^2 = (\abs{\nu(v)}^2 + \abs{\nu(jv)}^2) v^{\vee} \wedge (jv)^{\vee}.
\end{equation}
The latter is independent of the choice of $v \in T\Sigma_r\backslash\{0\}$. Using Stokes theorem, it is not difficult to see that
\begin{equation}
    E(u) = A_{H_{n_0}}(x_0) - \sum_{i=1}^d A_{H_{n_i}}(x_i) + \int_{\Sigma} R(\mathcal{K}_r),
\end{equation}
where $R(\mathcal{K}_r)$ is the curvature term from (\ref{curvature term}), and
\begin{equation}
    A_{H_t}(x) = -\int_{S^1} x^*\theta + \int_{S^1} H_t(x(t)) dt.
\end{equation}
Therefore, we get an \`a-priori energy estimate for solutions of \eqref{lollipop equation},
\begin{equation} \label{energy-estimate}
    E(u) \leq A_{H_{n_0}}(x_0) - \sum_{i=1}^d A_{H_{n_i}}(x_i) - C_{\bn}\abs{F}.
\end{equation}

In the theory of pseudo-holomorphic curves, Gromov compactness requires an \`a-priori $C^0$-estimate as well. This is proved using an integrated version of the maximum principle due to Abouzaid-Seidel, see \cite[Lemma 7.2]{open-string-analogue}.

\begin{lemma}
    Let $(r,u) \in \Mms^{\al}_{d,\bp,\bn}(x_0,\dots,x_d) $ be a solution of the pseudo-holomoprhic equation \eqref{lollipop equation}. Then $\emph{\text{Im}}(u)\subseteq \{r \leq 1-\epsilon\}$.
\end{lemma}
\begin{proof}
    Choose $r_0 \in (1-\epsilon,1)$ such that $u$ is transverse to $\{r=r_0\}$, and consider $S=u^{-1}(\{r\geq r_0\})$. Then $S$ is a compact Riemann surface with boundary, and $v = u|_{S}$ solves the equation
    \begin{equation}
        (dv - X_{h(r)} \otimes \gamma_r)_{J_r}^{0,1} = 0.
    \end{equation}

By Stokes theorem
\begin{align}
  0 \leq  E(v) &\defeq \frac{1}{2}\int_S \norm {(dv - X_{h(r)} \otimes \gamma_r)}^2\\
    &= \int_{S}v^*d\theta - a v^*dr \wedge \gamma_r \notag\\
    &= \int_{\partial S} \theta\circ dv - a r_0 \gamma_r + \int_S ard\gamma_r \notag\\
    &\leq \int_{\partial S}\theta\circ (dv - X_{h(r)}\otimes\gamma_r) \notag\\
    &= \int_{\partial S} \theta\circ J \circ (dv - X_H\otimes\gamma_r) \circ - j_S \notag\\
    &= \int_{\partial S} dr\circ dv \circ -j_S. \notag
\end{align}

However, if $\zeta \in T\partial S$ is positively oriented, $j\zeta$ must point inwards in $S$, so that $d(r\circ v)(j\zeta)\geq 0$. It follows that $E(v) = 0$ so that $dv = X_{h(r)}\otimes \gamma_r$ identically on $S$, and hence $u(S)\subseteq \{r\leq r_0\}$. The Lemma now follows by taking a limit $r_0 \rightarrow 1-\epsilon$.
\end{proof}

The previous lemma, combined with the energy estimate from (\ref{energy-estimate}), ensures that Gromov compactness holds. Denote by $\Mms^{\al,l}_{d,\bp,\bn}(\bx)$ the $l$-dimensional component of the moduli space $\Mms^{\al}_{d,\bp,\bn}(x_0,\dots,x_d)$.

\begin{corollary} \label{Gromov compactness}
    The moduli space $\Mms^{\al,l}_{d,\bp,\bn}(\bx)$ is a smooth manifold and admits a compactification given by 
    \begin{equation}
    \overline{\Mms}^{\al,l}_{d,\bp,\bn}(\bx) = \bigcup\ \{\Mms^{\al,l-\abs{iE(T)}}_{T,\mathbf{F},\bn}(\bx_e)\ | \ T\in \mathscr{T}(d)\}.
    \end{equation}
    Here, $\bx_e$ is a collection of Hamiltonian orbits $x_e \in \mathcal{X}_{n_e}$, one for each edge  $e\in E(T)$, which agrees with $x_0$ at the root, and with $x_i$ on  $i^{\text{th}}$-leaf for $i=1,\dots,d$.  When $l=1$, $\overline{\Mms}^{\al,1}_{d,\bp,\bn}(\bx)$ is a smooth manifold with boundary.
\end{corollary}

\subsection{The $L_{\infty}$ structure}\label{subsec: The L_infty-structure}

We now sketch the construction of the $L_{\infty}$ structure on symplectic cohomology. The details, including sign computations, are deferred to section \ref{Orientations}.

For each isolated point $\bu\in \Mms^{\al}_{d,\bp,\bn}(\bx)$, the linearized operator $D_u$ determines a linear map
\begin{equation}\label{contribution1}
    \abs{o_{\mathbf{u}}}: t^{i_1}\abs{o_{x_1}}\otimes\dots\otimes t^{i_d}\abs{o_{x_d}} \rightarrow \abs{o_{x_0}}[3-2d],
\end{equation}
where $i_j = \abs{\bp^{-1}(j)}$ for each $j\in\{1,\dots, d\}$. By adding the contributions of all isolated points $u \in \Mms^{\al}_{d,\bp,\bn}(\bx)$ for all possible $\bp$, $\bn$, and $\bx$, we obtain linear maps
\begin{equation}\label{0-th coeff}
    \tilde{\ell}_0^d: SC^*(X) ^{\otimes d}\rightarrow \bigoplus_{n=1}^{\infty} CF^*(X,H_{n}).
\end{equation}
Note that when $\bp$ is not injective, the moduli space $\Mms^{\al}_{d,\bp,\bn}(\bx)$ has an overall zero contribution, see Lemma \ref{lem:gradsym}.
The map \eqref{0-th coeff} admits a unique extension
\begin{equation}
    \tilde{\ell}^d: SC^*(X)^{\otimes d} \rightarrow SC^*(X)\ [3-2d]
\end{equation}
which commutes with $\partial_t$, see Lemma \ref{lem: commutes with partial-t}. By carefully examining the boundary strata of the Gromov compactifications of $1$-dimensional components of the moduli spaces $\Mms^{\al}_{d,\bp,\bn}(\bx)$, one sees that the maps $(\tilde{\ell}^d)_{d\geq 1}$ satisfy the $L_{\infty}$ relations
\begin{equation}\label{L-infinity equation}
    \sum_{\substack{1 \le j \le d\\ \sigma\in \Unsh(j,d)}} (-1)^\epsilon\tilde{\ell}^{d-j+1}(\tilde{\ell}^{j}(x_{\sigma(1)},\dots,x_{\sigma(j)}),x_{\sigma(j+1)},\dots,x_{\sigma(d)}) = 0.
 \end{equation}
We explain this equation in more detail, including the signs, in Section \ref{Orientations}.

The reader may notice that $\tilde{\ell}^1$ is not exactly the differential we previously defined for $SC^*(X)$: it is missing the term $-x'$ from \eqref{eq:dSC}, without which it does not compute the correct cohomology. To remedy this, define
\begin{equation}\label{eq:modified-elld}
    \ell^{d}(x_1,\dots,x_d) =
    \begin{cases*}
        \tilde{\ell}^d (x_1,\dots,x_d) & if $d\geq 2$\\
        \tilde{\ell}^1(x_1) - \partial_t x_1 & if $d = 1$.
    \end{cases*}
\end{equation}
These new operations also satisfy the $L_{\infty}$ relations, as we explain in Section \ref{Orientations}.

\subsection{The action filtration}
Each $1$-periodic orbit $x\in \mathcal{X}_n$ appearing in the construction of $SC^*(X)$ has an associated action
\begin{equation}
    A(x) = -\int_{S^1} x^*\theta + \int_{S^1} H_n(x(t)) dt.
\end{equation}
The differential $\ell^1 = d^{SC}$ respects the action filtration
\begin{equation}
    SC_{>A}(X) \subseteq SC(X),
\end{equation}
where $SC_{>A}(X)$ is the analogue of \eqref{telescope complex} which only uses $1$-periodic orbits of action greater than $A$. However, the higher $L_{\infty}$ operations do not respect the action filtration. For that reason, we define a quasi-isomorphic $L_{\infty}$ subalgebra
\begin{equation}
    SC_{\nu}(X) = \bigoplus_{n\geq\nu} CF^*(X,H_n)[t]
\end{equation}
for each positive integer $\nu \in \mathbb{N}$. This now carries a shifted filtration
\begin{equation}
    F_{\nu}^{>A}SC_{\nu}(X) = \bigoplus_{n\geq \nu} CF_{>A+\delta_{\nu}}^*(X,H_n)[t],
\end{equation}
where 
\begin{equation}
    \delta_{\nu} = - 2\widetilde{C}_{\nu} .
\end{equation}
Note that the shifts $\delta_{\nu}$ are positive real numbers, and that $\lim \delta_{\nu} = 0$.

\begin{lemma}
    The $L_{\infty}$ operations on $SC_{\nu}(X)$ respect the shifted action filtration
    \begin{equation}
        \ell^d:F_{\nu}^{>A_1}SC_{\nu}(X)\otimes\dots\otimes F_{\nu}^{>A_d}SC_{\nu}(X) \rightarrow F_{\nu}^{>A_1+\dots+A_d}SC_{\nu}(X).
    \end{equation}
\end{lemma}

\begin{proof}
    The proof amounts to showing that whenever a moduli space $\Mms^{\al}_{d,\bp,\bn}(\bx)$ contributes to the $L_{\infty}$ operations, we have
    \begin{equation}
        A(x_0) -\delta_{\nu} \geq \sum_{i=1}^d (A(x_i) -\delta_{\nu}).
    \end{equation}
    Equivalently, we need to show that
    \begin{equation}
        (d-1)\delta_{\nu} \geq -A(x_0) + \sum_{i=1}^d A(x_i).
    \end{equation}
    If $d=1$, this is obviously true. If $d\geq 2$, then using the energy estimate \eqref{energy-estimate}, it suffices to show that
    \begin{equation}
        (d-1)\delta_{\nu} \geq -C_{\bn}\abs{F}.
    \end{equation}
    The latter is true when $\abs{F}\leq d$ because
    \begin{equation}
        2(d-1)\geq d \geq \abs{F}\quad\text{and}\quad -\widetilde{C}_\nu\geq -C_{\bn}\geq 0.
    \end{equation}
    The case when $\abs{F}> d$ is irrelevant since the total contribution of $\Mms^{\al}_{d,\bp,\bn}(\bx)$ is $0$, see Lemma \ref{lem:gradsym}.
\end{proof}

\section{Wrapped Floer theory}\label{sec:wrapped}

\subsection{Lagrangian Floer theory}

Recall that a Lagrangian submanifold $L\subseteq X$ is said to be cylindrical if $\theta\vert_{L}\in \Omega^1(L)$ is an exact form with compact support in $L\cap\{r\leq 1-\epsilon\}$. We will \emph{always} assume that our Lagrangians are cylindrical, and that they carry an anchored brane structure as described in \S \ref{subsection:grading datum}. Fix a basic sequence $\bH'$ of Hamiltonians on $X$.

Let $(L_0,L_1)$ be a pair of cylindrical Lagrangians such that every $w$-chord $y\in \mathfrak{X}_w(L_0,L_1)$ is non-degenerate, where
\begin{equation}\label{chord-equation}
    \mathfrak{X}_w(L_0,L_1)= \{ y:[0,1]\rightarrow X\ \text{s.t}\ y(0)\in L_0, y(1)\in L_1 \ |\ \dot{y}(t) = X_{H'_{w}}(y(t)) \}.
\end{equation}
As before, each non-degenerate $y\in \mathfrak{X}_w(L_0,L_1)$ has an associated normalized orientation line $\abs{o_y}$. Floer's Lagrangian complex is 
\begin{equation}\label{Lagrangian-Floer-complex}
CF^*(L_0,L_1;H'_w) = \bigoplus_{y\in \mathfrak{X}_w(L_0,L_1)} \abs{o_y}.
\end{equation}
The differential counts solutions $u:\mathbb{R}\times [0,1] \rightarrow X$ of the pseudo-holomorphic equation
\begin{equation}\label{Lagrangian Floer trajectories}
    \begin{cases}
       & \partial_s u + J_{w} (\partial_t u - X_{H'_w}(u)) = 0,\\
       & u(-,0)\in L_0,\ u(-,1)\in L_1,\\
       & E(u)\defeq \int \abs{\partial_s u}^2 < +\infty,
    \end{cases}
\end{equation}
where $J_w$ is an appropriate choice of $t$-dependent almost complex structure. For each pair $y_-,y_+\in \mathfrak{X}_w(L_0,L_1)$ of \emph{distinct} Hamiltonian chords, let
\begin{equation}
    \Mms^\disc(y_-,y_+) = \{ u \text{ satisfies } \eqref{Lagrangian Floer trajectories} \text{ s.t. } \lim_{s\to \pm \infty} u(s,t) = y_\pm \}/\mathbb{R}.
\end{equation}
The differential on Floer's complex \eqref{Lagrangian-Floer-complex} is given by
\begin{equation*}
    d{y_+} = \sum_{y_-\neq y_+} \# \Mms^\disc(y_-,y_+) \cdot y_-,
\end{equation*}
where $\# \Mms^\disc(y_-,y_+)$ is the (signed) count of isolated points.

We can also define continuation chain maps
\begin{equation}\label{Lagrangian continuation map}
    \kappa: CF^*(L_0,L_1;H'_w) \to CF^*(L_0,L_1;H'_{w+1})
\end{equation}
by counting solutions $u:\mathbb{R}\times [0,1] \rightarrow X$ of the pseudo-holomorphic equation
\begin{equation}\label{Lagrangian continuation equation}
    \begin{cases}
       & \partial_s u + J_{s,t} (\partial_t u - X_{H'_{s,t}}(u)) = 0,\\
       & u(-,0)\in L_i,\ u(-,1)\in L_j,\\
       & E(u)\defeq \int \abs{\partial_s u}^2 < +\infty,
    \end{cases}
\end{equation}
where now $(J_{s,t},H'_{s,t})$ is a domain-dependent family interpolating between $(J_{w},H'_{w})$ (at $s\to +\infty$) and $(J_{w+1},H'_{w+1})$ (at $s\to -\infty$). 
For each pair $y_-,y_+\in \mathfrak{X}_w(L_0,L_1)$ of (not necessarily distinct) Hamiltonian chords, let
\begin{equation}
    \Mms^\al_c(y_-,y_+) = \{ u \text{ satisfies } \eqref{Lagrangian continuation equation} \text{ s.t. } \lim_{s\to \pm \infty} u(s,t) = y_\pm \}.
\end{equation}
The continuation map \eqref{Lagrangian continuation map} is
\begin{equation}
    \kappa(y_+) = \sum_{y_-} \# \Mms^\al_c(y_-,y_+) \cdot y_-,
\end{equation}
where $\# \Mms^\al_c(y_-,y_+)$ is the signed count of isolated points.

The wrapped Floer complex is the homotopy direct limit of the continuation maps \eqref{Lagrangian continuation map},
\begin{equation}\label{wrapped Floer complex}
    \matheu{W}(L_0,L_1) = \bigoplus_{w=1}^{\infty} CF^*(L_0,L_1;H'_w)[t].
\end{equation}
As before, $t$ is a formal variable of degree $-1$ such that $t^2 = 0$. The differential on this complex is
\begin{equation}\label{Wrapped-floer-complex-differential}
    \mu^1(y+ty') = \mu^1(y) -t\mu^1(y') + \kappa(y')-y' .
\end{equation}
Its cohomology is denoted by $HW^*(L_0,L_1)$, and it is called wrapped Floer cohomology. It is independent of the choice of the basic sequence $\bH'$.

\subsection{The wrapped Fukaya category}
Fix a countable collection $\mathbf{\mathcal{L}} = (L_i)_{i\in I}$ of cylindrical Lagrangians, and let $\bH'$ be a \emph{generic} basic sequence so that:
\begin{itemize}
    \item[-] For all distinct $i,j \in I$ and all integers $w\geq 0$, $L_i$ and $\phi^{X_{H'_w}}(L_j)$ intersect transversely.
    \item[-] There are no triple intersections amongst distinct Lagrangians of the form $\phi^{X_{H_{w_k}}}(L_i)$ with $w_k\geq 0$ and $i\in I$.
\end{itemize}

Consider a tuple of cylindrical Lagrangians $L_0,\dots,L_k \in \mathcal{L}$, a disc $r = [(C,\zeta_0,\dots,\zeta_k,z_1,\dots,z_d),\psi] \in \Rms^\disc_{d,k,\bp,\bw}$, and maps $u:\Sigma_r \rightarrow X$ in the complement $\Sigma_r = C\backslash \{\zeta_0,\dots,\zeta_k,z_1,\dots,z_d\}$ which satisfy Floer's equation,
\begin{equation}\label{popsicle equation}
    \begin{cases}
        & (du-Y_r)_{J_r}^{0,1} = 0,\\
        & E(u) \defeq \int \abs{du - Y_r}^2 < +\infty,\\
        & u((\partial\Sigma_r)_i) \subseteq L_i.
    \end{cases}
\end{equation}
The boundary components of $\partial\Sigma_r$ have been labelled by $i\in\{0,\dots,k\}$ so that $(\partial\Sigma_r)_i$ originates from $\zeta_i$ with respect to the boundary orientation. As in \S \ref{Hamiltonian Floer theory}, $Y_r$ and $J_r$ are appropriate choices of domain-dependent Hamiltonian perturbations and almost complex structures.

Write the tuple of weights as $\bw = (w_0,\dots,w_k,n_1,\dots,n_d)$, let $\bx = (x_1,\dots,x_d)$ be a collection of $1$-periodic $\bH$-orbits such that $x_1\in \mathcal{X}_{n_1},\dots, x_d\in\mathcal{X}_{n_d}$, and let $\by = (y_0,\dots,y_k)$ be a collection of $\bH'$-chords such that $y_0\in \mathfrak{X}_{w_0}(L_0,L_k), y_1\in \mathfrak{X}_{w_1}(L_0,L_1),\dots, y_k\in \mathfrak{X}_{w_k}(L_{k-1},L_k)$. Then, we have a moduli space $\Mms^\disc_{r}(\bx,\by)$ consisting of solutions to \eqref{popsicle equation} which satisfy the constraints
\begin{equation}
    \begin{cases}
        & \lim_{s\rightarrow \pm\infty} \epsilon_{\zeta_i}^*u(s,t) = y_{i}, \\
        & \lim_{s\rightarrow \pm\infty} \epsilon_{z_i}^*u(s,t) = x_{i}.
    \end{cases}
\end{equation}

The family version allows for $r = [(C,\zeta_0,\dots,\zeta_k,z_1,\dots,z_d),\psi]$ to vary in its moduli space,
\begin{equation}
    \Mms^\disc_{d,k,\bp,\bw}(\bx,\by) = \{ (r,u) \ \vert \ r\in \Rms^\disc_{d,k,\bp,\bw} \ \text{and} \ u \in \Mms^\disc_{r}(\bx,\by) \}.
\end{equation}

We also use the notation
\begin{equation}
    \Mms^\disc_{k,\bp,\bw}(\by) \defeq \Mms^\disc_{0,k,\bp,\bw}(\emptyset,\by).
\end{equation}

Following the same approach of \S \ref{subsec:Transversality and compactness}, we can produce a consistent universal choice of contact type almost complex structures $J_r$, and Hamiltonian perturbations $K'_r = H'_r\otimes \gamma_r$ such that the moduli spaces $\Mms^\disc_{k,\bp,\bw}(\by)$ are smooth finite-dimensional manifolds. The only modification we need to account for is that now $\gamma_r$ is required to vanish when restricted to the boundary $\partial \Sigma_r$, see \cite[\S 2.5 , \S 2.6]{open-string-analogue}.

Its Gromov compactification is
\begin{equation}
    \overline{\Mms}^\disc_{k,\bp,\bw}(\by) = \bigcup\ \{\overline{\Mms}^\disc_{T,\bF,\bw}(\by_e) \ |\ T\in \mathscr{T}^{\ord}(k) \}.
\end{equation}

Similarly, the same approach of \S \ref{subsec: The L-infty relations} shows that the count of isolated points in $\Mms^\disc_{k,\bp,\bw}(\by)$ produces a linear map
\begin{equation}
     CF^*(L_{0},L_{1};H'_{w_1})[t] \otimes \dots \otimes  CF^*(L_{k-1},L_k;H'_{w_k})[t]\rightarrow CF^*(L_{0},L_{k}; H'_{w_0}).
\end{equation}
These maps can be uniquely extended to
\begin{equation}
   \tilde{\mu}^k: \matheu{W}(L_0,L_1)\otimes\dots\otimes \matheu{W}(L_{k-1},L_{k}) \rightarrow \matheu{W}(L_0,L_k)\ [2-d]
\end{equation}
which are $\partial_t$-equivariant (see \eqref{eq:delt_eq} for the precise formulation of what this means). See \S \ref{subsec: The L-infty relations} for a geometric interpretation of this extension.

By examining the Gromov compactification of $1$-dimensional moduli spaces, one sees that the maps $\tilde{\mu}^{k}$ satsify the $A_{\infty}$ relations. Again, the differential $\tilde{\mu}^{1}$ is different from the one we have previously defined. Let
\begin{equation}
    \mu^{k}(y_1,\dots,y_k) =
    \begin{cases*}
        \tilde{\mu}^k (y_1,\dots,y_k) & if $k\geq 2$\\
        \tilde{\mu}^1(y_1) - \partial_t y_1 & if $k = 1$.
    \end{cases*}
\end{equation}
Then, the operations $\mu^k$ still satisfy the $A_{\infty}$ relations, and $(\matheu{W},\mu^k)$ is called the wrapped Fukaya category on the objects $(L_i)_{i\in I}$.

\subsection{The closed open map}\label{subsec:co}

The $A_{\infty}$ category $\matheu{W}$ has an associated Hochschild co-chain complex
\begin{equation}
    \CC^*(\matheu{W}) = \prod_{L_0,\dots,L_k \in \mathcal{L}}\hom_{\bk}(\matheu{W}(L_0,L_1)\otimes\dots\otimes\matheu{W}(L_{k-1},L_k),\matheu{W}(L_0,L_k)[-k]).
\end{equation}
The Hochschild complex comes with a Gerstenhaber product
\begin{align*}
    (\phi\circ \psi)^k (a_k,\dots,a_1) &\defeq \sum_{i,j} (-1)^{\dagger} \phi^{k-j+1}(a_{k},\dots,\psi^{j}(a_{i+j},\dots,a_{i+1}), \dots, a_1),\quad\text{where}\\
\dagger &= (\deg \psi - 1)(\sum_{l=1}^i \deg a_l - i).
\end{align*}
     It satisfies $\deg \phi \circ \psi = \deg \phi + \deg \psi - 1$, and can be used to endow $CC^*(\matheu{W})[1]$ with a graded Lie algebra structure
\begin{equation}
    [\phi,\psi] = \phi \circ \psi - (-1)^{(\deg \psi -1)(\deg \phi-1)} \psi\circ \phi.
\end{equation}

Our discussion so far does not involve the $A_{\infty}$ structure on $\matheu{W}$. In fact, the latter is the same as the choice of an element $\mu \in \CC^2(\matheu{W})$ such that $\mu^0 = 0$ and $\mu\circ\mu = 0$. Moreover, an $A_{\infty}$ structure on $\matheu{W}$ can be used to define a differential on the Hochschild complex,
\begin{equation}
    \partial \phi = [\mu,\phi].
\end{equation}
This turns $\CC^*(\matheu{W})[1]$ into a differential graded Lie algebra (dgLa). Its cohomology is known as Hochschild cohomology
\begin{equation}
    HH^*(\matheu{W}) = H^*(\CC^*(\matheu{W}),\partial).
\end{equation}

We now observe that a dgLa defines an $L_\infty$ algebra in the sense of \cite{LadaMarkl}, by setting 
$$\ell^1_{\mathrm{LM}}(x) = \partial x,\qquad \ell^2_{\mathrm{LM}}(x,y) = [x,y], \qquad \ell^{\ge 3}_{\mathrm{LM}} = 0.$$
Thus we may define an $L_\infty$ structure in our sense using equation \eqref{eq:LM-trans}. 
In fact the maps $\CO^d$ define an $L_\infty$ homomorphism to the negation of this $L_\infty$ structure (it is clear that if $\ell^d$ satisfy the $L_\infty$ relations then so do $-\ell^d$, as the relations are quadratic in the $\ell^*$). 
Explicitly, this negated $L_\infty$ structure on $\CC^*(\matheu{W})$ is defined by:
\begin{equation}
    \label{eq:L-inf-cc}
\ell^1(\phi) = -\partial \phi,\qquad \ell^2(\phi,\psi) = (-1)^{1+\deg \phi} [\phi,\psi],\qquad \ell^{\ge 3} = 0. 
\end{equation}

More generally, one can construct Hamiltonian perturbations and contact type almost complex structures on $\Sigma_r$ for all $r\in \Rms^\disc_{d,k,\bp,\bw}$ so that $\Mms^\disc_{d,k,\bp,\bw}(\bx,\by)$ are smooth finite dimensional manifolds with Gromov compactifications
\begin{equation}
    \overline{\Mms}^\disc_{d,k,\bp,\bw}(\bx,\by) = \bigcup\ \{\overline{\Mms}^\disc_{T,\bF,\bw}(\bx_e,\by_e) \ |\ T\in \mathscr{T}^{\cl}(d,k)\}.
\end{equation}

The count of isolated points in $\Mms^\disc_{d,k,\bp,\bw}(\bx,\by)$ produces a linear map
\begin{equation}\label{co contributions}
   co_{d,k,\bp,\bw}:t^{i_1}\abs{o_{x_1}}\otimes \dots t^{i_d}\abs{o_{x_d}} \otimes t^{j_1}\abs{o_{y_1}} \otimes \dots \otimes t^{j_k}\abs{o_{y_k}}   \rightarrow \abs{o_{y_0}}[2-2d-k],
\end{equation}
where  $i_p = \abs{\bp^{-1}(p)}$ and $j_p = \abs{\bp^{-1}(d+p)}$. The sum of all the contributions \eqref{co contributions} produces a linear map
\begin{equation}\label{eq:cokd}
    co_{d,k}: SC^*(X)^{\otimes d}\otimes \matheu{W}(L_0,L_1) \otimes\dots\otimes\matheu{W}(L_{k-1},L_k) \rightarrow \bigoplus_{w=1}^{\infty} CF^*(L_0,L_k;H'_w)[2-2d-k].
\end{equation}
This map is graded symmetric in the inputs from $SC$, in the sense that
\begin{equation}
    \label{eq:cokd_symm}
    co_{d,k}(x_{\sigma(1)},\ldots,x_{\sigma(d)};y_1,\ldots,y_k) = (-1)^\epsilon co_{d,k}(x_{1},\ldots,x_{d};y_1,\ldots,y_k).
\end{equation}
It can be extended uniquely to a $\partial_t$-equivariant map
\begin{equation}
    co_{d,k}: SC^*(X)^{\otimes d}\otimes \matheu{W}(L_0,L_1) \otimes\dots\otimes\matheu{W}(L_{k-1},L_k) \rightarrow \matheu{W}(L_0,L_k),
\end{equation}
where $\partial_t$-equivariance means
\begin{multline}\label{eq:delt_eq}
        \partial_t co_{d,k}(x_1,\ldots,x_d;y_1,\ldots,y_k) = \sum_{j=1}^d (-1)^\dagger co_{d,k}(x_1,\ldots,\partial_t x_j,\ldots,x_d;y_1,\ldots,y_k) +\\
        \sum_{j=1}^k (-1)^* co_{d,k}(x_1,\ldots,x_d;y_1,\ldots,\partial_t y_j,\ldots,y_k)
\end{multline}
where
\begin{align*}
    \dagger &= 1+{\sum_{l=1}^{j-1}|x_l|},\\
    * &= 1+{\sum_{l=1}^d |x_l| + \sum_{l=1}^{j-1}|y_l|'}.
\end{align*}
The extension remains graded symmetric.

Putting the maps $co_{d,k}$ together, we obtain maps
\begin{equation}
    \CO^{d} : SC^*(X)^{\otimes d} \rightarrow \CC^*(\matheu{W})[2-2d].
\end{equation}
By definition, the element $\CO^0 \in \CC^2(\matheu{W})$ coincides with the element $\tilde \mu$. 

By examining the boundary strata of the Gromov compactification, one sees that we have algebraic relations
\begin{multline}\label{eq:CO-rel}
    \sum_{\substack{1 \le j \le d\\ \sigma \in \Unsh(j,d)}} (-1)^{\epsilon} co_{d-j+1,k}(\tilde \ell^j(x_{\sigma(1)},\ldots,x_{\sigma(j)}),x_{\sigma(j+1)},\ldots,x_{\sigma(d)};y_1,\ldots,y_k) + \\
    \sum_{\substack{0 \le j \le d \\ 0\le k_1 < k_2 \le k \\
    \sigma \in \Unsh(j,d)}} (-1)^\dagger co_{d-j+1,k_1}(x_{\sigma(1)},\ldots,x_{\sigma(j)};y_1,    \ldots,y_{k_1},co_{j,k_2-k_1}(x_{\sigma(j+1)},\ldots \\
    \ldots ,x_{\sigma(d)};y_{k_1+1},\ldots,y_{k_2}),y_{k_2+1},\ldots,y_k) = 0,
\end{multline}   
where 
$$\dagger = \epsilon + \left(1+\sum_{i=j+1}^d |x_{\sigma(i)}| \right)\left(\sum_{i=1}^{k_1} |y_i|'\right) + \sum_{i=1}^j |x_{\sigma(i)}| .$$
We refer to Section \ref{Orientations} for a brief discussion of the signs. 
Note that for $d=0$, this reduces to the $A_\infty$ relations $\tilde \mu \circ \tilde \mu = 0$.

By $\partial_t$-equivariance of $co_{d,k}$, we find that \eqref{eq:CO-rel} continues to hold if we replace $\tilde \ell$ with $\ell$ and $co_{0,k} = \tilde \mu^k$ with $\mu^k$. 
From the definition of the Gerstenhaber bracket, this can be rewritten as
\begin{multline}\label{eq:CO-rel-2}
    \sum_{\substack{1 \le j \le d\\ \sigma \in \Unsh(j,d)}} (-1)^{\epsilon} \CO^{d-j+1}(\ell^j(x_{\sigma(1)},\ldots,x_{\sigma(j)}),x_{\sigma(j+1)},\ldots,x_{\sigma(d)}) +  [\mu , \CO^d(x_1,\ldots,x_d)] + \\
    \sum_{\substack{1 \le j \le d-1\\ \sigma \in \Unsh(j,d) \\ \sigma(1)<\sigma(j+1)}} (-1)^{\epsilon+\sum_{i=1}^j|x_{\sigma(i)}|} \left[ \CO^j(x_{\sigma(1)},\ldots,x_{\sigma(j)}),\CO^{d-j}(x_{\sigma(j+1)},\ldots,x_{\sigma(d)})\right] = 0.
\end{multline}   

This equation (together with graded symmetry of the maps $\CO^d$) means that
\begin{equation}\label{CO map}
    \CO : (SC^*(X,\bH)[1], \ell^*) \rightarrow (\CC^*(\matheu{W})[1],\partial,[-,-]).
\end{equation}
is an $L_{\infty}$ morphism. 

To establish the connection with the conventional notion of $L_\infty$ morphism, we recall the definition. 
An $L_\infty$ morphism from one $L_\infty$ algebra $(V_0,\ell_0^d)$ to another $(V_1,\ell_1^d)$ consists of maps $F^d: V_0^{\otimes d} \to V_1[2-2d]$ for $d \ge 1$, which are graded symmetric in the sense that
    $$F^d(x_{\sigma(1)},\ldots,x_{\sigma(d)}) = (-1)^\epsilon F^d(x_1,\ldots,x_d),$$
    and furthermore satisfy
    \begin{multline}
        \sum_{\substack{1 \le j \le d\\ \sigma \in \Unsh(j,d)}} (-1)^\epsilon F^{d-j+1}(\ell_0^j(x_{\sigma(1)},\ldots,x_{\sigma(j)}),x_{\sigma(j+1)},\ldots,x_{\sigma(d)}) = \\
        \sum_{\substack{1 \le j_1 < j_2 < \ldots < j_k = d\\ \sigma \in \Unsh^<(j_1,\ldots,j_k)}} (-1)^\epsilon \ell_1^k(F^{j_1}(x_{\sigma(1)},\ldots,x_{\sigma(j_1)}),F^{j_2-j_1}(\ldots),\ldots,F^{j_k-j_{k-1}}(\ldots,x_{\sigma(d)})).
    \end{multline}
    Here $\Unsh(j_1,\ldots,j_k)$ is the set of $(j_1,\ldots,j_k)$-unshuffles, that is, permutations $\sigma \in \mathfrak{S}_{j_k}$ which satisfy
    $$\sigma(j_{i-1}+1) < \sigma(j_{i-1}+2) < \ldots < \sigma(j_i) \qquad \text{for all $1 \le i \le k$ (where $j_0 = 0$ by convention);}$$
    and $\Unsh^{<}(j_1,\ldots,j_k) \subset \Unsh(j_1,\ldots,j_k)$ is the subset of unshuffles such that
    $$\sigma(1)<\sigma(j_1+1) < \sigma(j_2+1) < \ldots < \sigma(j_{k-1}+1).$$

This is compatible with the conventional definition: if $F^d$ is an $L_\infty$ morphism in our sense, then we can define $F^d_{\mathrm{LM}}$ by the straightforward analogue of the formula \eqref{eq:LM-trans}. 
Then $F^d_{\mathrm{LM}}$ will be an $L_\infty$ morphism from $(V_0,\ell^d_{0,\mathrm{LM}})$ to $(V_1,\ell^d_{1,\mathrm{LM}})$ in the sense of \cite[Definition 2.3]{Allocca}.

It is now straightforward to check that the maps $\CO^d$ define an $L_\infty$ morphism, where the $L_\infty$ structure on Hochschild cochains is given in \eqref{eq:L-inf-cc}.

\section{Orientations}\label{Orientations}

\subsection{Conventions} Throughout this section, we work with the following conventions. If $D$ is a Fredholm operator, we let $\det(D)$ be its determinant line, placed in degree $\ind(D)$. 
Similarly, if $V$ is a finite dimensional (ungraded) vector space of dimension $n$, we set
\begin{equation}
    \lambda(V) \defeq \wedge^{n}(V[-1]) \cong \wedge^{n} V [-n].
\end{equation}
Recall that given two $\mathbb{G}$-graded lines $o_{x_1}$ and $o_{x_2}$, we have a Koszul isomorphism
\begin{equation}
    o_{x_1}\otimes o_{x_2} \rightarrow o_{x_2}\otimes o_{x_1}:\quad x_1\otimes x_2 \mapsto (-1)^{\abs{x_1}\abs{x_2}} x_2\otimes x_1.
\end{equation}
More generally, we have a Koszul isomorphism for any $\sigma \in \mathfrak{S}_d$,
\begin{equation}
    o_{x_1}\otimes\dots \otimes o_{x_d} \xrightarrow{\text{Koszul}} o_{x_{\sigma(1)}}\otimes\dots\otimes o_{x_{\sigma(d)}},
\end{equation}
which differs from the tautological permutation of inputs by a sign $(-1)^{\epsilon}$, where we recall that we use the abbreviation
\begin{equation}\label{epsilon sign}
    \epsilon = \sum_{\substack{i<j \\ \sigma(i)>\sigma(j)}} \abs{x_i} \cdot \abs{x_j}
\end{equation}
throughout the paper. 
The main goal of this section is to construct maps 
\begin{equation}
    \ell^d: SC(X)^{\otimes d} \rightarrow SC(X) [3-2d]
\end{equation}
which are symmetric 
\begin{equation}\label{L-infty symmetric}
     \ell ^d(x_{\sigma(1)},\ldots,x_{\sigma(d)}) = (-1)^\epsilon \tilde \ell^d(x_1,\ldots,x_d)
\end{equation}
and satisfy the $L_\infty$ relations
\begin{equation}\label{L-infty equation}
    \sum_{\substack{1 \le j \le d\\ \sigma \in \Unsh(j,d)}} (-1)^{\epsilon} \ell^{d-j+1}(\ell^j(x_{\sigma(1)},\dots,x_{\sigma(j)}),x_{\sigma(j+1)},\dots,x_{\sigma(d)}) = 0.
\end{equation}

\subsection{Orienting the domains}\label{sec:or_dom}

Let $\bp: F \rightarrow \{1,\dots,d\}$ be a flavour on an ordered set $F$. Recall that the moduli space $\Rms^{\al}_{d,\bp}$ can be thought of as a quotient
\begin{equation}
    \Rms^\al_{d,\bp}  \cong \left(\mbox{Conf}_d(\mathbb{C}) \times \mathbb{R}^F\right)/\mbox{Aff}(\mathbb{C},\mathbb{R}_{>0}).
\end{equation}
Therefore, we get a short exact sequence
\begin{equation} \label{orientation sequence for lollipops}
    0 \to T\mbox{Aff}(\mathbb{C},\mathbb{R}_{>0}) \to T(\mathbb{C}^d \times \mathbb{R}^F) \to T\Rms^{\al}_{d,\bp}\to 0.
\end{equation}
Note that $T\mbox{Aff}(\mathbb{C},\mathbb{R}_{>0}) \cong \mathbb{C} \oplus \mathbb{R}$, where the $\mathbb{C}$-factor corresponds to translations and the $\mathbb{R}$-factor to positive scalings. Thus we have an isomorphism 
\begin{equation}
    \lambda(T\mbox{Aff}(\mathbb{C},\mathbb{R}_{>0})) \cong \lambda(\mathbb{C})\otimes T_s \cong \mathbb{R}[-2] \otimes \Ts_s,
\end{equation}
where $\Ts_s$ is a trivial $\mathbb{G}$-graded line in degree $1$. Triviality means that it is equipped with an isomorphism $\Ts_s\cong \mathbb{R}[-1]$. Similarly,
\begin{equation}
    \lambda(T(\mathbb{C}^d \times \mathbb{R}^F)) \cong \mathbb{R}[-2d]\otimes \bigotimes_{f\in F} \Ts_f,
\end{equation}
where $\Ts_f$ is a trivial $\mathbb{G}$-graded line in degree $1$ for each $f\in F$. It follows from the short exact sequence \eqref{orientation sequence for lollipops} that 
\begin{equation}\label{orientation of domains}
    \lambda(T\Rms^{\al}_{d,\bp}) \cong \mathbb{R}[-2d+2]\otimes \Ts_F \otimes \Ts_s^{\vee},
\end{equation}
where $\Ts_F = \bigotimes_{f\in F} \Ts_f$. Recall that the group
\begin{equation}
    \Sym(d,\bp) \defeq \{(\sigma,\pi) \in \Sym(d) \times \Sym(F): \bp(\pi(f)) = \sigma(\bp(f))\}    
\end{equation}
acts on $\Rms^{\al}_{d,\bp}$. An element $(\sigma,\pi)$ acts on the orientation \eqref{orientation of domains} with a sign equal to the signature of the permutation $\pi \in \Sym(F)$.

Recall that the boundary strata of $\overline{\Rms}^{\al}_{d,\bp}$ are indexed by data $(T,\bF)$. The differential of the gluing map $\bar{\gamma}_{T,\bF}$ defines an isomorphism
\begin{equation}
    D\bar{\gamma}_{T,\bF}: T(\Rms^{\al}_{T,\bF}) \times T\left([0,\epsilon)^{\iE(T)}\right) \to T(\overline{\Rms}^{\al}_{d,\bp}).   
\end{equation}
Let $\{\Ts_e\}_{e \in \iE(T)}$ be $\mathbb{G}$-graded lines in degree $1$. Then, the gluing map $\bar{\gamma}_{T,\mathbf{F}}$ induces an isomorphism
\begin{equation}\label{gluing map orientation} 
    \bigotimes_{v \in V(T)} \lambda(T\Rms^{\al}_{d_v,\bp_v}) \otimes \bigotimes_{e \in \iE(T)} \Ts_e \cong \lambda(T\overline{\Rms}^{\al}_{d,\bp})
\end{equation}
which identifies $\Ts_e$ with the tangent space to the corresponding copy of $[0,\epsilon)$. Moreover, we have the following commutative diagram
\begin{equation}\label{eq:gluing_orientations}
 \begin{tikzcd}
    \Bigotimes_{v\in V(T)}\lambda(T\Rms^{\al}_{d_v,\bp_v}) \Bigotimes_{e \in \iE(T)} \Ts_e \arrow{r}{\eqref{orientation of domains}} \arrow[swap]{d}{\eqref{gluing map orientation}} & \Bigotimes_{v \in V(T)} \left(\mathbb{R}[-2d_v+2] \otimes\Ts_{F_v}\otimes\Ts_{s_v}^{\vee}\right)  \Bigotimes_{e \in \iE(T)} \Ts_e \arrow{d}{} \\
    \lambda(T\Rms^{\al}_{d,\bp}) \arrow{r}{\eqref{orientation of domains}}& \mathbb{R}[-2d+2]\otimes \Ts_F\otimes \Ts_s^{\vee}.
 \end{tikzcd}
\end{equation}
Here, the right vertical map is induced by isomorphisms 
\begin{eqnarray*}
    \Ts_f & \cong & \Ts_f \mbox{ for $f \in F_v \subseteq F$} \\
    \Ts_{s_v} & \cong & \Ts_e \mbox{ when $e \in \iE(T)$ is the incoming edge of $v \in V(T)$} \\
    \Ts_{s_v} & \cong & \Ts_s \mbox{ when $v \in V(T)$ is the root vertex}
\end{eqnarray*}
which are compatible with the given trivializations.

\subsection{The $L_{\infty}$ operations revisited}\label{subsec: L-infinity operations in detail}

We now expand on the description of the $L_{\infty}$ structure described in \S \ref{subsec: The L_infty-structure}. Let $\mathbf{n}=(n_0,\dots,n_d)$ be a set of weights, and let $\bp:F\rightarrow \{1,\dots,d\}$ and $\bq: G\rightarrow \{1,\dots,d\}$ be two flavours such that
\begin{equation}\label{flavour constraint}
    \begin{cases}
        & n_0 = n_1 + \dots + n_d + \abs{F},\\
        & \abs{\bp^{-1}(j)}\cdot \abs{\bq^{-1}(j)} = 0 \quad \text{for}  \ j=1,\dots,d.
    \end{cases}    
\end{equation}
The flavour $\bq$ is used to track the multiplicity of $t$ in the output. For each collection $\mathbf{x}=(x_0,\dots,x_d)$ of orbits $x_i \in \mathcal{X}_{H_{n_i}}$, define 
\begin{equation}
    \Mms^\al_{d,\bp,\bq,\bn}(\bx) \defeq \Mms^\al_{d,\bp,\bn}(x_0,\dots,x_d).
\end{equation}

For each regular element $\bu \defeq (r,u) \in \Mms^\al_{d,\bp,\bq,\bn}(\bx)$, we have a Cauchy-Riemann operator $D_u$ associated with $u$. It is obtained by linearizing the pseudo-holomorphic equation \eqref{lollipop equation} at $u$ while keeping $r$ constant. Using the gluing theorems explained in \cite[\S (11c)]{PL-theory}, the Fredholm operator $D_u$ satisfies the  gluing isomorphism
\begin{equation}\label{gluing iso 1}
    o_{x_1}\otimes \dots \otimes o_{x_d} \otimes \lambda(D_u) \cong o_{x_0}.
\end{equation}
Since $\bu$ is regular, we also have an isomorphism
\begin{equation}\label{tangent space iso} 
    \lambda(T_r \Rms^{\al}_{d,\bp})\otimes \lambda(D_u) \cong \lambda(T_\bu \Mms^{\al}_{d,\bp,\bq,\bn}(\bx)).
\end{equation}

Now, define trivial $\mathbb{G}$-graded lines $\Ts_f$ and $\Ts_g$ in degree $1$ for all $f \in F$ and $g \in G$ as before, and for each $i=1,\dots,d$, set $F_i := \bp^{-1}(i)$, $G_i = \bq^{-1}(i)$, and
\begin{equation}
    o^{FG}_i := \Ts_{G_i}^{\vee} \otimes  \Ts_{F_i}^{\vee} \otimes o_{x_i},\quad \text{where } \Ts_A = \bigotimes_{a\in A} \Ts_a.
\end{equation}
Note that we have a preferred isomorphism $o^{FG}_j\cong t^{i_j}o_{x_j}$, where $i_j = \abs{F_j} + \abs{G_j}$ and $t$ is a formal variable of degree $-1$. Similarly, we set $o^G_0 = T^{\vee}_G\otimes o_{x_0} \cong t^{i_0} o_{x_0}$, where $i_0 = \abs{G}$. The contribution of $\bu$ to the $L_{\infty}$ operations is an isomorphism
\begin{equation} \label{contribution of u to L-infty ops} 
    o_\bu: t^{i_1} \cdot o_{x_1} \otimes \ldots \otimes t^{i_d} \cdot o_{x_d} \xrightarrow{\cong} t^{i_0} \cdot o_{x_0}[3-2d],
    \end{equation}
which we now explain. We start with the following composition of isomorphisms
\begin{align}\label{main iso}
    o^{FG}_1\otimes\dots\otimes o^{FG}_d 
    &\xrightarrow{\text{Koszul}} T_G^{\vee}\otimes o_{x_1}\otimes\dotsi\otimes o_{x_d}\otimes T_F^{\vee} \\
    &\xrightarrow{\eqref{gluing iso 1}} T^{\vee}_G \otimes o_{x_0} \otimes \lambda(D_u)^{\vee} \otimes T_F^{\vee} \notag\\
    &\xrightarrow{\eqref{tangent space iso}} T^\vee_G \otimes o_{x_0} \otimes \lambda(T_\bu \Mms^{\al}_{d,\bp,\bq,\bn}(\bx))^{\vee} \otimes \lambda(T_r \Rms^{\al}_{d,\bp})\otimes T_F^{\vee} \notag \\
    &\xrightarrow{\eqref{orientation of domains}} T^\vee_G \otimes o_{x_0} \otimes \lambda(T_\bu \Mms^{\al}_{d,\bp,\bq,\bn}(\bx))^{\vee} \otimes T_s^{\vee} \otimes \mathbb{R}[2-2d] \notag\\
    & \xrightarrow{\text{Koszul}} T^\vee_G \otimes T_s^{\vee} \otimes o_{x_0} \otimes \lambda(T_\bu \Mms^{\al}_{d,\bp,\bq,\bn}(\bx))^{\vee} \otimes \mathbb{R}[2-2d]. \notag
\end{align}
When $\bu \in \Mms^{\al}_{d,\bp,\bq,\bn}(\bx)$ is isolated, we obtain a tautological isomorphism $\lambda(T_\bu \Mms^{\al}_{d,\bp,\bq,\bn}(\bx)) \cong \mathbb{R}[0]$ which, together with \eqref{main iso} and the given trivialization $\Ts_s \cong \mathbb{R}[-1]$, produces the contribution of $\bu$ in \eqref{contribution of u to L-infty ops}.

The sum of the normalizations of $o_\bu$ defines maps
\begin{equation}\label{lambda-map}
    \lambda^{d,\bn}_{\bp,\bq}(\bx) = \sum_{\bu \in \Mms^{\al,0}_{d,\bp,\bq,\bn}(\bx) } \abs{o_\bu}.
\end{equation}

\begin{lemma}\label{lem:gradsym}
The maps $\lambda^{d,\bn}_{\bp,\bq}(\bx)$ are graded symmetric, which means that for any permutation $\sigma \in \mathfrak{S}_d$,
\begin{equation}
    \lambda^{d,\bn}_{\bp,\bq}(\bx)(x_{\sigma(1)},\ldots,x_{\sigma(d)}) = (-1)^\epsilon \cdot \lambda^{d,\bn}_{\bp,\bq}(\bx)(x_1,\ldots,x_d),    
\end{equation}
where $\epsilon$ is the Koszul sign of the action of $\sigma$ on $(x_1,\dots,x_d)$ from \eqref{epsilon sign}. 
Furthermore, if $\abs{\bp^{-1}(i)} \geq 2$ for some $i$, then $\lambda^{d,\bn}_{\bp,\bq}(\bx) = 0$.  
\end{lemma}

\begin{proof}
Note that $\bq:G\to \{1,\ldots,d\}$ is irrelevant, so we might as well assume that $G=\emptyset$ to keep notation simple. For any $(\sigma,\pi) \in \Sym(d,\bp)$, the diagram
\begin{equation}
\begin{tikzcd}
    \Ts_{F_1} \otimes o_{x_1} \otimes \ldots \otimes \Ts_{F_d}\otimes o_{x_d} \arrow{r}{o_\bu} \arrow{d}{\text{Koszul}} &  o_{x_0}[3-2d] \arrow{d}{} \\
    \Ts_{\pi \cdot F_{\sigma(1)}} \otimes o_{x_{\sigma(1)}} \otimes \ldots \otimes \Ts_{\pi\cdot F_{\sigma(d)}} \otimes o_{x_{\sigma(d)}} \arrow{r}{o_{(\sigma,\pi)\cdot\bu}}& o_{x_0}[3-2d]
\end{tikzcd}
\end{equation}
commutes. The second statement follows by applying the commutative diagram when $\sigma = \id$ and $\pi$ is a transposition of two elements in $\bp^{-1}(i)$. Indeed, it shows that $\bu$ and $(\sigma,\pi) \cdot \bu$ contribute with opposite sign, and hence the terms in the sum defining $\lambda^{d,\bn}_{\bp,\bq}(\bx)$ come in cancelling pairs. Finally, the first statement of the lemma follows by applying the commutative diagram when $\pi = \id$.
\end{proof}

By adding all the maps $\lambda^{d,\bn}_{\bp,\bq}$ for all possible flavours $\bp$ and $\bq$, we obtain linear maps
\begin{equation}\label{n-piece}
    \tilde{\ell}^{d,\bn}: CF^*(X,H_{n_1})[t] \otimes\dotsi\otimes CF^*(X,H_{n_d})[t] \rightarrow CF^*(X,H_{n_0})[t],
\end{equation} 
where $t$ is now a formal parameter of degree $-1$ and $t^2 = 0$. 

\begin{lemma}\label{lem: commutes with partial-t}
Let $j \in G$ such that $\bq(j) = k$. Define $G' := G \setminus \{j\}$, and $\bq':= \bq_{|G'}$. 
Suppose $\bu \in \Mms^\al_{d,\bp,\bq,\bn}(\bx)$, and $\bu' \in \Mms^\al_{d,\bp,\bq',\bn}(\bx)$ is the corresponding element. Then the diagram
\begin{equation}
 \begin{tikzcd}
    t^{i_1}\cdot o_{x_1} \otimes \ldots \otimes t^{i_k} \cdot o_{x_k} \otimes \ldots \otimes t^{i_d} \cdot o_{x_d} \arrow{r}{o_\bu} \arrow{d} &  t^{i_0} \cdot o_{x_0} \arrow{d}{} \\
    t^{i_1} \cdot o_{x_1} \otimes \ldots  \otimes t^{i_k-1} \cdot o_{x_k} \otimes \ldots \otimes t^{i_d} \cdot o_{x_d} \arrow{r}{o_{\bu'}}& t^{i_0-1} \cdot o_{x_0}
 \end{tikzcd}
\end{equation}
is commutative up to the sign $(-1)^{\Delta}$, where $\Delta = 1+\abs{t^{i_{k-1}}x_{k-1}}+\dotsi+\abs{t^{i_1}x_1}$.
\end{lemma}

\begin{proof}
The isomorphisms of the top and bottom arrows differ only in the first Koszul isomorphism of \eqref{main iso}. The difference is whether one commutes $\Ts_j^{\vee}$ (appearing in $o^{FG}_k$) through $o^{FG}_{k-1},\dots, o^{FG}_{1}$ and $\Ts_s^\vee$ before trivializing it. The sign difference is the Koszul sign $(-1)^{\Delta}$.
\end{proof}

\begin{corollary}\label{cor:lambdadeltmod}
The maps 
\begin{equation}
    \tilde{\ell}^{d,\bn} \defeq \sum_{\bx,\bp,\bq} \lambda^{d,\bn}_{\bp,\bq}(\bx)    
\end{equation}
are $\Z[\partial_t]$-module maps, i.e.
\begin{equation}
    \partial_t \tilde{\ell}^{d,\bn}(x_1,\ldots,x_d) = \sum_{j=1}^d (-1)^{1+\abs{x_1} + \dotsi + \abs{x_{j-1}} } \cdot \tilde{\ell}^{d,\bn}(x_1,\ldots,\partial_t x_j,\ldots,x_d).   
\end{equation}    
\end{corollary}

\subsection{The $L_\infty$ relations} \label{subsec: The L-infty relations}

By adding the maps $\tilde{\ell}^{d,\bn}$ for all possible weights $\bn$, we get linear maps
\begin{equation}
    \tilde{\ell}^d: SC(X)^{\otimes d} \rightarrow SC(X) [3-2d].
\end{equation}
We can now show that these maps satisfy the $L_{\infty}$ relations. Recall that the boundary components of the $1$-dimensional moduli spaces $\Mms^{\al,1}_{d,\bp,\bn}(\bx)$ have a decomposition
\begin{equation}\label{codimension 1 boundary}
    \partial\overline{\Mms}^{\al,1}_{d,\bp,\bn}(\bx) = \bigsqcup_{(T,\mathbf{F})\ \text{s.s.}}\Mms^{\al,0}_{T,\mathbf{F},\bn}(\bx_e).
\end{equation}
The tree $T$ must have exactly one internal edge, hence $T$ has two vertices which we denote $v_+$ and $v_-$, where $v_+$ is closest to the root. Let $d_{\pm}$ be the number of outgoing edges from $v_\pm$. The leaves of $T$ which are adjacent to $v_\pm$ are labelled by ordered, disjoint, and complementary subsets $S_\pm\subseteq \{1,\dots,d\}$. Note that $\abs{S_-}=d_-$ and $\abs{S_+}=d_+-1$ since one of the outgoing edges of $v_+$ is an internal edge and not a leaf.
\begin{figure}
    \centering
    \begin{tikzpicture}[x=0.75pt,y=0.75pt,yscale=-.8,xscale=1]
    \draw    (160,184) -- (287,185) ;
    \draw    (287,185) -- (347,235) ;
    \draw    (287,185) -- (457,135) ;
    \draw    (287,185) -- (455,45) ;
    \draw    (347,235) -- (463,195) ;
    \draw    (347,235) -- (455,296) ;
    \draw    (347,235) -- (453,327) ;

    \draw (271,165.4) node [anchor=north west][inner sep=0.75pt]    {$v_{+}$};
    \draw (325,234.4) node [anchor=north west][inner sep=0.75pt]    {$v_{-}$};
    \draw (319,193.4) node [anchor=north west][inner sep=0.75pt]    {$e$};
    \draw (456,318.4) node [anchor=north west][inner sep=0.75pt]    {$s_{1}^{-}$};
    \draw (459,285.4) node [anchor=north west][inner sep=0.75pt]    {$s_{2}^{-}$};
    \draw (465,181.4) node [anchor=north west][inner sep=0.75pt]    {$s_{d_{-}}^{-}$};
    \draw (462,219) node [anchor=north west][inner sep=0.75pt]   [align=left] {.\\.\\.};
    \draw (459,126.4) node [anchor=north west][inner sep=0.75pt]    {$s_{1}^{+}$};
    \draw (458,29.4) node [anchor=north west][inner sep=0.75pt]    {$s_{d_{+} -1}^{+}$};
    \draw (463,62) node [anchor=north west][inner sep=0.75pt]   [align=left] {.\\.\\.};
    \end{tikzpicture}
    \caption{}
    \label{fig:my_label}
\end{figure}

For each pair $d_\pm$ satisfying $d_+ + d_- = d + 1$, the number of possible $d$-leafed trees $T$ with $1$ internal edge such that $\deg(v_\pm) = d_\pm$ is $d!/(d_-!(d_+-1)!)$, which is the count of partitions $(S_+,S_-)$. Each such partition dictates the weights and Hamiltonian orbits associated with the vertices $v_\pm$ as follows,
\begin{align}
    \bn_+ &= (n_0, n_e, n_{s^+_1},\dots,n_{s^+_{d_+-1}}) \quad\text{and}\quad \bn_- = (n_e,n_{s^-_1},\dots,n_{s^-_{d_-}}) \\
    \bx_+ &= (x_0, x_e ,x_{s^+_1},\dots,x_{s^+_{d_+-1}}) \quad\quad\quad\quad \bx_- = (x_e, x_{s^-_1},\dots,x_{s^-_{d_-}}),
\end{align}
where $S_+ = \{ s^+_1<\dots <s^+_{d_+-1}\}$, $S_- = \{ s^-_1<\dots< s^-_{d_-}\}$, and $e$ is the unique internal edge of $T$.

Let $\bp_{\pm}$ be the induced flavours from $\bp$ on $F_{v_\pm}$. Then
\begin{equation}
    \Mms^{\al,0}_{T,\mathbf{F},\bn}(\bx_e) = \Mms^{\al,0}_{d_+,\bp_+,\bn_+}(\bx_+)\times \Mms^{\al,0}_{d_-,\bp_-,\bn_-}(\bx_-).
\end{equation}

Given a new map $\bq:G\rightarrow \{1,\dots,d\}$, there is a way of breaking $\bq$ into flavours $\bq_{+}: G_+\rightarrow S_{+}\cup\{e\}$ and $\bq_-: G_-\rightarrow S_-$ for each of the vertices of $T$. We take $G_+ = G$ with the flavour
\begin{equation}
    \bq_+(g) =
    \begin{cases*}
        e & $\text{if}\ \bq(g)\in S_-,$\\
        \bq(g) &$\text{otherwise}.$
    \end{cases*}
\end{equation}

Similarly, we take $G_- = \bp_+^{-1}(e)\sqcup \bq_+^{-1}(e)$ with the flavour
\begin{equation}
    \bq_-(g) =
    \begin{cases*}
        \bp(g) & $\text{if}\ g\in \bp_+^{-1}(e),$\\
        \bq(g) & $\text{if}\ g\in \bq_+^{-1}(e)$.
    \end{cases*}
\end{equation}

Note that the pairs $(\bp_\pm,\bq_\pm)$ obtained this way are \emph{composable}, in the sense that
\begin{equation}
    \abs{G_-} = \abs{\bp_+^{-1}(e)} + \abs{\bq_+^{-1}(e)}.
\end{equation}
More importantly, any composable pair $(\bp_\pm,\bq_\pm)$ arises from the construction above for an appropriate $\bq:G\rightarrow \{1,\dots, d\}$. Fix such a map $\bq$ and consider a pair $(\bu_+,\bu_-)$ in the resulting codimension $1$ boundary stratification,
\begin{equation}\label{boundary decomposition}
    \partial\overline{\Mms}^{\al,1}_{d,\bp,\bq,\bn}(\bx) = \bigsqcup_{T,\mathbf{F}} \Mms^{\al,0}_{d_+,\bp_+,\bq_+,\bn_+}(\bx_+)\times \Mms^{\al,0}_{d_-,\bp_-,\bq_-,\bn_-}(\bx_-).
\end{equation}
For ease of notation, we set
\begin{equation}
    \Mms^{\al}(\bx) = \overline{\Mms}^{\al,1}_{d,\bp,\bq,\bn}(\bx)\quad\text{and}\quad \Mms_{\pm}^{\al}(\bx_\pm) = \Mms^{\al,0}_{d_\pm,\bp_\pm,\bq_\pm,\bn_\pm}(\bx_\pm).
\end{equation}
Then, we have a gluing isomoprhism 
\begin{equation}\label{gluing iso 2}
    \lambda(T_{\bu_+}\Mms_+^{\al}(\bx_+)) \otimes \lambda(T_{\bu_-}\Mms_-^{\al}(\bx_-)) \otimes \Ts_e \cong \lambda(T_{\bu}\Mms^{\al}(\bx)),
\end{equation}
where $\bu$ is the gluing of $\bu_+$ and $\bu_-$ with small gluing parameter $\delta$, and $\Ts_{e}\cong \mathbb{R}[-1]$ is a trivial line which is identified with the tangent line $\lambda(T_{\delta}[0,\epsilon))$, oriented in the direction of increasing $\delta$.

Using the isomorphism of \eqref{main iso}, we have
\begin{align}
    o_{s_1^{-}}^{(FG)_-}\otimes\dots \otimes o_{s_{d_-}^{-}}^{(FG)_-} &\rightarrow \Ts_{s_-}^{\vee}\otimes o_e^{G_-} \otimes \lambda(T_{\bu_-}\Mms_-^{\al}(\bx_-))^{\vee}\otimes  \mathbb{R}[2-2d_-]\label{iso 5}\\
    o_{e}^{(FG)_+}\otimes\dots \otimes o_{s_{d_+-1}^{+}}^{(FG)_+} &\rightarrow \Ts_{s_+}^{\vee}\otimes o_0^{G_+} \otimes \lambda(T_{\bu_+}\Mms_+^{\al}(\bx_+))^{\vee}\otimes  \mathbb{R}[2-2d_+].\label{iso 6}
\end{align}
After composing these two isomorphisms, we get
\begin{align}\label{iso 7}
    &o_{s_1^{-}}^{(FG)_-}\otimes\dots \otimes o_{s_{d_-}^{-}}^{(FG)_-}\otimes  o_{s_1^+}^{(FG)_+}\otimes\dots \otimes o_{s_{d_+-1}^{+}}^{(FG)_+}\\
    &\xrightarrow{\text{Kzl}_1+ \eqref{iso 5}} \Ts_{s_-}^{\vee} \otimes  o_e^{G_-}\otimes o_{s_1^+}^{(FG)_+}\otimes\dots \otimes o_{s_{d_+-1}^{+}}^{(FG)_+} \otimes \lambda(T_{\bu_-}\Mms_-^{\al}(\bx_-))^{\vee}[2-2d_-]\notag\\
    &\xrightarrow{\eqref{iso 6} +\text{Kzl}_2} \Ts_{s_-}^{\vee} \otimes \Ts_{s_+}^{\vee} \otimes o_0^{G_+} \otimes\lambda(T_{\bu_+}\Mms_+^{\al}(\bx_+))^{\vee} \otimes \lambda(T_{\bu_-}\Mms_-^{\al}(\bx_-))^{\vee} [2-2d].\notag\\
    &\xrightarrow{\eqref{gluing iso 2} +\text{Kzl}_3} \Ts_{s_-}^{\vee}\otimes \Ts_{s_+}^{\vee} \otimes o_0^{G_+} \otimes \lambda(T_{\bu}\Mms^{\al}(\bx))^{\vee}\otimes \Ts_e[2-2d].\notag
\end{align}
At the same time, we could use also \eqref{main iso} for $\bu$ and directly obtain an isomorphism
\begin{align}\label{iso 8}
    o_1^{FG}\otimes\dotsi\otimes o_d^{FG}
    &\xrightarrow{\text{Kzl}_{4}}o_{s_1^{-}}^{(FG)_-}\otimes\dots \otimes o_{s_{d_-}^{-}}^{(FG)_-}\otimes  o_{s_1^+}^{(FG)_+}\otimes\dots \otimes o_{s_{d_+-1}^{+}}^{(FG)_+} \\
    &\xrightarrow{\eqref{main iso} + \text{Kzl}_5} \Ts_{s}^{\vee} \otimes o_0^{G_+} \otimes \lambda(T_{\bu}\Mms^{\al}(\bx))^{\vee} \otimes \mathbb{R}[2-2d].
\end{align}
After using the given trivializations of $\Ts_s$, $\Ts_{s_\pm}$ and $\Ts_e$, the isomorphisms \eqref{iso 7} and \eqref{iso 8} agree up to a sign $(-1)^{\S}$, where $\S = \sum_{i=1}^5 \text{Kzl}_i$ is a sum of Koszul signs. We can compute that
\begin{equation}\label{boundary signs}
    \S = 1 + \epsilon(\sigma;x_1,\dots,x_d),
\end{equation}
where $\sigma$ is the unshuffle $(s^-_1,\dots,s^-_{d_-},s^{+}_{1},\dots,s^{+}_{d_+-1})$. Since the line $\Ts_e$ is trivialized in the inward pointing direction of the $1$-dimensional moduli space $\Mms^{\al}(\bx)$, which is smooth and compact, we deduce that the sum of the contributions $\abs{o_{\bu_+} \circ \left(o_{\bu_-} \otimes \id^{\otimes(d_+-1)}\right)}$ of the boundary points, weighted by their associated signs \eqref{boundary signs}, is $0$. It follows that the operations $\tilde \ell^d$ satisfy the $L_{\infty}$ equation \eqref{L-infty equation} as claimed. 

It is immediate that the modified $L_\infty$ operations $\ell^d$, defined in \eqref{eq:modified-elld}, are symmetric. 
By combining the $L_\infty$ relations for $\tilde \ell^d$ with Lemma \ref{lem: commutes with partial-t}, it is immediate that $\ell^d$ also satisfy the $L_\infty$ equation.

\subsection{Wrapped Fukaya category and closed--open map}

We now sketch how to define the maps $co_{d,k}$, and prove that they satisfy \eqref{eq:CO-rel}. 
This amounts to combining the arguments above with those from \cite[Appendix B]{Sheridan-versality}.

We first observe that 
$$\Rms^\disc_{d,k} = \mathrm{Conf}_{d,k}(\mathbb{H})/\mathrm{Aff}(\mathbb{H})$$ 
is the quotient of the space of configurations of $k$ ordered boundary points and $d$ interior points on the upper half-plane $\mathbb{H}$ by the group of affine automorphisms $\mathrm{Aff}(\mathbb{H}) = \{z \mapsto az+b,a \in \R_{>0},b\in \R\}$. 
As in Section \ref{sec:or_dom}, this gives an isomorphism
$$\lambda(T \Rms^\disc_{d,k,\bp}) = \R[-2d] \otimes \Ts_F \otimes \Ts_s^\vee \otimes \Ts_t^\vee \otimes \bigotimes_{j=1}^k \Ts_j,$$
where $\Ts_s$ corresponds to the scaling factor in $\mathrm{Aff}(\mathbb{H})$ and $\Ts_t$ corresponds to the translation factor. 
These are trivialized by taking the direction of increasing $a$, respectively $b$.

Recall that boundary strata of $\overline{\Rms}^\disc_{d,k,\bp}$ are indexed by data $(T,\bF)$. 
The gluing map near a boundary stratum induces an isomorphism
$$\bigotimes_{v \in V^{\text{solid}}(T)} \lambda(T\Rms^\disc_{d_v,k_v,\bp_v}) \otimes \bigotimes_{v \in V^{\text{dash}}(T)} \lambda(T \Rms^\al_{d_v,\bp_v}) \otimes \bigotimes_{e \in \text{iE}(T)} \Ts_e \cong \lambda(T\Rms^\disc_{d,k,\bp}).$$
The analogue of the commutative diagram \eqref{eq:gluing_orientations} now says that the resulting isomorphism
\begin{multline*}
    \bigotimes_{v \in V^{\text{solid}}(T)} \R[-2d_v] \otimes \Ts_{F_v} \otimes \Ts_{s_v}^\vee \otimes \Ts_{t_v}^\vee \otimes \bigotimes_{j=1}^{k_v}\Ts_{j_v} \otimes \bigotimes_{v \in V^{\text{dash}}(T)} \R[2-2d_v] \otimes \Ts_{F_v} \otimes \Ts_{s_v}^\vee \otimes \bigotimes_{e \in \text{iE}(T)} \Ts_e \\
    \cong \R[-2d] \otimes \Ts_F \otimes \Ts_s^\vee \otimes \Ts_t \otimes \bigotimes_{j=1}^k \Ts_j
\end{multline*}
is induced by Koszul sign rule and natural identifications of trivialized lines, together with identifications (respecting the given trivializations)
\begin{eqnarray*}
    \Ts_{s_v} & \cong & \Ts_e \mbox{ when $e \in \iE(T)$ is the incoming edge of $v \in V(T)$} \\
    \Ts_{t_v} & \cong &\Ts_{j_w} \mbox{ when the incoming solid edge of $v$ is the $j_w$th outgoing solid edge of $w$}\\
    \Ts_{s_v} & \cong  &\Ts_s  \mbox{ when $v \in V(T)$ is the root vertex}\\
    \Ts_{t_v} & \cong &\Ts_t \mbox{ when $v \in V(T)$ is the root vertex}, \\
\end{eqnarray*}
and the natural trivialization $\Ts^\vee \otimes \Ts \cong \R$ for any line $\Ts$.

Now, let $\bw = (n_1,\ldots,n_d,w_0,\ldots,w_k)$ be a set of weights; $\bx$ a corresponding tuple of orbits and $\by$ a corresponding tuple of chords; and $\bp:F \to \{1,\ldots,k+d\}$ and $\bq:G \to \{1,\ldots,k+d\}$ be flavours such that 
\begin{equation*}
\begin{cases}
    &w_0 = n_1+\ldots+n_d+w_1+\ldots+w_k + \abs{F},    \\
    &\abs{\bp^{-1}(j)} \cdot \abs{\bq^{-1}(j)} = 0 \mbox{ for $j=1,\ldots,k+d$.}
\end{cases}    
\end{equation*}
Define $F_{x_p} \defeq \bp^{-1}(p)$ and $F_{y_p} \defeq \bp^{-1}(d+p)$; and define $G_{x_p}$ and $G_{y_p}$ similarly, but using $\bq$ instead of $\bp$.
Set
$$o_{x_p}^{FG} \defeq \Ts_{G_{x_p}}^\vee \otimes \Ts_{F_{x_p}}^\vee \otimes o_{x_p} \cong t^{i_p} \cdot o_{x_p},$$
where $i_p = \abs{F_{x_p}} + \abs{G_{x_p}}$. 
Set 
$$o_{y_p}^{FG} \defeq \Ts_{G_{y_p}}^\vee \otimes \Ts_p^\vee \cong t^{j_p}\cdot o_{y_p}[1]$$
for $1 \le p \le k$; for $p=0$ we define
$$o_{y_0}^G \defeq \Ts_G^\vee \otimes o_{y_0} \cong t^{j_0} \cdot o_{y_0},$$
where $j_0 = \abs{G}$. 
We now explain how a rigid element $\bu \in \Mms^\disc_{d,k,\bp,\bq,\bw}(\bx,\by) \defeq \Mms^\disc_{d,k,\bp,\bw}(\bx,\by)$ determines an isomorphism 
$$o_\bu: t^{i_1}\cdot o_{x_1} \otimes \ldots \otimes t^{i_d}\cdot o_{x_d} \otimes t^{j_1}\cdot o_{y_1}[1] \otimes \ldots \otimes t^{j_k}\cdot o_{y_k}[1] \xrightarrow{\sim} t^{j_0} \cdot o_{y_0}[2-2d].$$
We have isomorphisms
\begin{align*}
    o^{FG}_{x_1} \otimes \ldots \otimes o^{FG}_{x_d} \otimes o^{FG}_{y_1} \otimes \ldots \otimes o^{FG}_{y_k} & \cong \Ts_G^\vee \otimes o_{x_1} \otimes \ldots \otimes o_{x_d} \otimes o_{y_1} \otimes \ldots \otimes o_{y_k} \otimes \Ts_F^\vee \otimes \bigotimes_{p=1}^k \Ts_p^\vee \\
    &\cong \Ts_G^\vee \otimes o_{y_0} \otimes \lambda(D_u)^\vee \otimes \Ts_F^\vee \otimes \bigotimes_{p=1}^k \Ts_p^\vee\\
    &\cong \Ts_G^\vee \otimes o_{y_0} \otimes \lambda(T_u\Mms^\disc_{d,k,\bp}(\bx,\by))^\vee \otimes \lambda(T_r\Rms^\disc_{d,k,\bp}) \otimes \Ts_F^\vee \otimes \bigotimes_{p=1}^k \Ts_p^\vee\\
    & \cong \Ts_G^\vee \otimes \Ts_s^\vee \otimes o_{y_0} \otimes \Ts_t^\vee \otimes \R[-2d],
\end{align*}
whose definitions are parallel to those in \eqref{main iso}. 
Applying the trivializations of $\Ts_s$ and $\Ts_t$, we obtain the desired isomorphism $o_{\bu}$. 

Summing the normalizations of the maps $o_{\bu}$ defines the maps $co_{d,k}$ from \eqref{eq:cokd}, parallel to the definition of $\tilde \ell^d$. 
The proofs that the maps $co_{d,k}$ are graded symmetric in the entries $x_j$ (in the sense of \eqref{eq:cokd_symm}), $\partial_t$-equivariant (in the sense of \eqref{eq:delt_eq}), and satisfy the $L_\infty$ relations \eqref{eq:CO-rel}, are parallel to the corresponding proofs for the maps $\tilde \ell^d$. 
In particular, all signs appearing in these equations arise from the Koszul sign rule. 
This is all that was required for the construction of the $A_\infty$ structure on the wrapped Fukaya category, and the $L_\infty$ homomorphism $\CO$ (see Section \ref{subsec:co}).

\bibliographystyle{alpha}
\bibliography{bibliography.bib}
\end{document}